\documentclass[english]{jnsao}
\usepackage{enumitem}
\usepackage{booktabs}
\usepackage{graphicx}

\usepackage{tikz}
\usepackage{pgfplots}
\usepackage{pgfplotstable}
\pgfplotsset{compat=1.18}

\newcommand{\R}{\mathbb{R}}
\newcommand{\N}{\mathbb{N}}

\newcommand{\st}{ \ \middle| \ }
\newcommand{\abs}[1]{|#1|}
\newcommand{\norm}[1]{\|#1\|}

\manuscripteprinttype{arXiv}
\manuscripteprint{2609.26645v2}

\title{Orlicz space relaxation of total variation denoising}

\author{Christian Clason%
    \thanks{Department of Mathematics and Scientific Computing, University of Graz, Austria, \email{c.clason@uni-graz.at}, \orcid{0000-0002-9948-8426}}
    \and
    Tobias Unterberger%
    \thanks{Institute of Analysis and Scientific Computing, TU Wien, Austria, \email{tobias.unterberger@tuwien.ac.at}, \orcid{0009-0004-0495-4693}}
}
\shortauthor{Clason, Unterberger}

\funding{%
    This research was funded in whole or in part by the Austrian Science
    Fund (FWF) \href{https://dx.doi.org/10.55776/F100800}{10.55776/F100800}.
}


\begin{document}

\maketitle

\begin{abstract}
    We consider variational image denoising with spatially dependent
    Orlicz regularization and introduce an Orlicz--Sobolev approximation
    of the Rudin--Osher--Fatemi (ROF) model based on a scaled $L\log L$-type density. For a
    general class of uniformly superlinear integrands satisfying a
    $\Delta_2$-condition, we establish existence and uniqueness of
    minimizers and derive a Fenchel dual problem with dual attainment
    and pointwise optimality conditions. We then specialize the framework
    to a logarithmic density whose correction is activated above a
    prescribed local gradient scale. For this model, we obtain explicit
    expressions of the Fenchel dual and optimality conditions as well
    as a pointwise radial formula for the dual proximal map involving
    the Lambert $W$-function. As the logarithmic parameter tends to zero,
    we prove equicoercivity and $\Gamma$-convergence to the ROF functional,
    together with convergence of the corresponding minimizers. Numerical
    illustrations indicate that the logarithmic correction can reduce
    staircasing for diffuse transitions, while retaining behavior
    comparable to ROF for sharp interfaces and a natural image.
\end{abstract}

\section{Introduction}
\label{s:1}

Total variation regularization is a standard tool in variational image
processing, since it suppresses noise while allowing for discontinuities
representing edges. Given a noisy image $f\in L^2(\Omega)$ and a
regularization parameter $\alpha>0$, the Rudin--Osher--Fatemi (ROF) model
\cite{RudinOsherFatemi1992} consists in minimizing
\[
    \mathcal F_0(u)
    =
    \frac12\norm{u-f}_{L^2(\Omega)}^2
    +
    \alpha\abs{Du}(\Omega)
\]
over $BV(\Omega)\cap L^2(\Omega)$. Here $BV(\Omega)$ denotes the
space of functions of bounded variation, whose distributional
derivatives $Du$ are finite Radon measures with total variation
$\abs{Du}(\Omega)$; see
\cite[Chapter~3]{ambrosio2000functions}.
The one-homogeneous growth of the regularization term permits
singular derivatives and thus avoids smoothing across jump
discontinuities as strongly as quadratic Sobolev regularization.
At the same time, total variation regularization tends to favor piecewise
constant reconstructions and may create artificial staircasing in regions
containing smooth or diffuse transitions
\cite{ChambolleLions1997,Jalalzai2016}. This has motivated, among other
approaches, higher-order regularizers such as total generalized variation
\cite{BrediesKunischPock2010}.

A different possibility, aimed at reducing staircasing in smooth or
diffuse regions while retaining a first-order model, is to replace the
linear growth of total variation by a suitably chosen non-standard growth
condition. This leads naturally to the question of whether one can
introduce a superlinear correction that discourages the concentration of
diffuse transitions into narrow regions of large gradient, while remaining
sufficiently close to linear growth to recover total variation in a
suitable limiting regime. Orlicz spaces provide a natural framework for
such models, since their defining modulars allow growth laws beyond fixed
powers; see, for instance,
\cite{rao1991theory,harjulehto2019generalized}. Their generalized,
spatially dependent counterparts additionally permit the regularization
behavior to vary across the image. Spatially dependent first-order energies
have previously been considered through variable-exponent models
\cite{ChenLevineRao2006} and double-phase regularization
\cite{HarjulehtoHasto2021}. More recently, a general framework for bounded
variation spaces with generalized Orlicz growth was developed in
\cite{EleuteriHarjulehtoHasto2025}.

Motivated by this question, we introduce and analyze a spatially scaled
$L\log L$-type regularization that provides a superlinear approximation
of total variation. Let
\[
    \kappa\in L^\infty(\Omega),
    \qquad
    0<\kappa_-\leq\kappa(x)\leq\kappa_+<\infty
    \quad\text{for a.e. }x\in\Omega,
\]
and, for $\gamma>0$, define
\[
    \varphi_\gamma(x,t)
    :=
    \alpha t
    \left(
        1+\gamma\log^+\frac{t}{\kappa(x)}
    \right),
    \qquad t\geq0,
\]
where $\log^+r:=\max\{0,\log r\}$ for $r>0$, with
$\log^+0:=0$.
(We omit here the dependence on $\kappa$ for the sake of presentation.)
We then consider the corresponding functional
\begin{equation}
    \mathcal F_\gamma(u)
    =
    \frac12\norm{u-f}_{L^2(\Omega)}^2
    +
    \int_\Omega
    \varphi_\gamma(x,\abs{\nabla u(x)})\,dx.
    \label{E: introduction scaled logarithmic functional}
\end{equation}
The scale function $\kappa$ determines the gradient magnitude at which the
logarithmic correction becomes active. More precisely,
$\varphi_\gamma(x,t)=\alpha t$ for $0\leq t\leq\kappa(x)$, whereas gradients
above this local scale receive an additional logarithmic penalty. Thus the
model retains the total-variation density for smaller gradients while
discouraging the concentration of a diffuse transition into narrow regions
of large gradient. The parameter $\gamma$ controls the strength of this
correction.

For every fixed $\gamma>0$, the density $\varphi_\gamma$ has superlinear
growth of $L\log L$-type, uniformly in the spatial variable. The
corresponding minimization problem is naturally posed in an
Orlicz--Sobolev space. Functions with finite energy belong to
$W^{1,1}(\Omega)$, so their distributional derivatives have no singular
part. In particular, jump discontinuities are excluded for fixed
$\gamma>0$, whereas they are admissible in the ROF model.
On the other hand,
\[
    \varphi_\gamma(x,t)\longrightarrow\alpha t
    \qquad\text{as }\gamma\downarrow0
\]
locally uniformly in $t$ and uniformly with respect to $x$. This suggests
viewing the family
\eqref{E: introduction scaled logarithmic functional}
as an Orlicz--Sobolev approximation of the ROF functional.

The functional-analytic treatment of this model requires some care.
Although the scaled logarithmic density satisfies a uniform
$\Delta_2$-condition, its conjugate has exponential growth and does not
satisfy a corresponding $\Delta_2$-condition. In particular, the natural
Orlicz spaces need not be reflexive, and the full dual Orlicz space must be
distinguished from the finite-modular domain of the conjugate functional.
While the general theory of convex integral functionals on Orlicz spaces
is classical, see, e.g., \cite{rockafellar1968integrals,kozek1980convex},
these distinctions are essential for obtaining the correct dual problem
for the present model.

The central result of this work is thus a rigorous Fenchel duality theory for
the resulting denoising problem in a setting where the conjugate integrand
fails the $\Delta_2$-condition and the natural spaces are not reflexive.
This yields a dual problem with attainment and a pointwise optimality
system, which in turn justifies an explicit primal--dual algorithm for the
scaled logarithmic model.
Specifically, we make the following contributions:

\noindent\emph{Well-posedness and duality for superlinear Orlicz regularization.}
We first establish a convex-analytic framework for functionals of the form
\[
    \mathcal F(u)
    =
    \frac12\norm{u-f}_{L^2(\Omega)}^2
    +
    \int_\Omega
    \varphi(x,\abs{\nabla u(x)})\,dx,
\]
where the integrand $\varphi(x,t)$ depends measurably on the spatial
variable $x$, has superlinear growth in $t$ uniformly in $x$, and
satisfies a $\Delta_2$-condition in $t$ with constants independent of
$x$. Using uniform integrability and weak compactness in $L^1$, we
prove existence of a minimizer without relying on reflexivity of the
underlying Orlicz space; the strict convexity of the fidelity term
yields uniqueness. Building on the classical characterization of the
dual space and of the Fenchel conjugate of the modular
\cite{rockafellar1968integrals,kozek1980convex,harjulehto2019generalized},
we derive the dual problem and optimality system, in which the
homogeneous Neumann boundary condition is encoded in a weak
integration-by-parts identity.

\noindent\emph{Explicit dual formulation and primal--dual realization of the logarithmic model.}
For the scaled logarithmic density, the scalar conjugate and the
relevant subdifferentials can be computed explicitly, which gives a
concrete dual formulation and optimality system. As in ROF denoising,
the proximal map of the fidelity term is explicit and the dual
proximal map acts pointwise and radially; its radial profile admits an
explicit piecewise formula involving the principal branch of the
Lambert $W$-function. The model therefore retains the computational
structure of the primal--dual method of \cite{ChambollePock2011},
despite the additional nonlinear growth.

\noindent\emph{Consistency with ROF denoising.}
We prove that the functionals $\mathcal F_\gamma$ are equicoercive and
$\Gamma$-converge to the ROF functional $\mathcal F_0$ with respect to
strong convergence in $L^1(\Omega)$ as $\gamma\downarrow0$. Together
with uniqueness of the ROF minimizer, this yields convergence of the
minimizers to the ROF solution and makes rigorous the interpretation of
\eqref{E: introduction scaled logarithmic functional} as an
Orlicz--Sobolev approximation of the ROF model.

\noindent\emph{Numerical illustrations.}
We compare the scaled logarithmic model at positive $\gamma$ with ROF
and quadratic $H^1$ regularization. The blurred-disk experiment
indicates that suitable parameter choices can reduce staircasing in
diffuse transitions, while the nearly sharp disk and natural-image
examples yield reconstructions close to ROF.

Let us relate our work to the most relevant literature.
Bounded variation spaces with generalized Orlicz growth were studied in
\cite{EleuteriHarjulehtoHasto2025}, covering variable exponent and
double phase models. There, the main difficulty arises at points where the
integrand has linear growth, so that the derivative may have a singular
part. The corresponding modular is defined in dual form, as a supremum
over vector fields involving the conjugate integrand; it is characterized
in terms of a spatially dependent recession function and obtained as the
$\Gamma$-limit of uniformly convex approximating energies. For superlinear
growth as considered here, the energy is finite only if the derivative has
no singular part. Our analysis therefore concerns this superlinear regime.
While conjugate duality is used in \cite{EleuteriHarjulehtoHasto2025} to
represent and evaluate the regularization functional, we apply it to the
denoising problem itself: we derive its Fenchel dual problem, prove dual
attainment, and obtain optimality conditions coupling primal and dual
solutions, which form the basis of the primal--dual numerical realization.
Furthermore, the scaled logarithmic density is in a sense the reverse of
the model of \cite{ChenLevineRao2006}, which uses superlinear growth for
small gradients to reduce staircasing and linear growth for large gradients
to preserve edges. Here, the total-variation density is kept for small
gradients and the superlinear correction acts on large gradients. The price
is that jumps are excluded for every $\gamma>0$, making the consistency result
as $\gamma\downarrow0$ important as a justification for the model.

This work is organized as follows. The general framework is developed in
\cref{s:3}. \Cref{s:4} applies it to the scaled logarithmic model and contains
the $\Gamma$-convergence analysis, and \cref{s:5} presents the numerical
implementation and illustrations. The paper concludes with a discussion of the
results and further directions.


\section{Variational analysis of convex Orlicz regularization}
\label{s:3}

\subsection{Notation and functional setting}
\label{s:notation}

Unless stated otherwise, $\Omega\subset\R^N$ is a bounded Lipschitz
domain with outer unit normal $\nu$. In the imaging applications
motivating the paper, $N=2$. We write
$\R_\infty:=\R\cup\{+\infty\}$. Identities and inequalities between
measurable functions are understood almost everywhere. Divergence
is understood in the distributional sense unless additional
regularity is specified. For background on generalized Orlicz
spaces, we refer to \cite[Chapter~3]{harjulehto2019generalized}.

Let
\[
    \varphi:\Omega\times[0,\infty)\to[0,\infty)
\]
be such that $x\mapsto\varphi(x,t)$ is measurable for every $t\geq0$,
and $t\mapsto\varphi(x,t)$ is convex, increasing, continuous, and
satisfies $\varphi(x,0)=0$ for almost every $x\in\Omega$.
We assume that $\varphi(\cdot,t)\in L^1(\Omega)$ for every $t\geq0$.
We further assume the uniform $\Delta_2$-condition
\[
    \varphi(x,2t)\leq C_\Delta\varphi(x,t)
    \qquad
    \text{for a.e. }x\in\Omega,\quad t\geq0,
\]
for some $C_\Delta\geq2$, and uniform superlinear growth:
\[
    \lim_{t\to\infty}
    \operatorname*{ess\,inf}_{x\in\Omega}
    \frac{\varphi(x,t)}{t}
    =\infty.
\]

For $z\in\R^N$, set $\Phi(x,z):=\varphi(x,\abs{z})$.
The associated modular is
\[
    \rho_\varphi(v)
    :=
    \int_\Omega\varphi(x,\abs{v(x)})\,dx
    =
    \int_\Omega\Phi(x,v(x))\,dx
\]
for measurable vector fields $v:\Omega\to\R^N$.
The generalized Orlicz space is
\[
    L^\varphi(\Omega;\R^N)
    :=
    \left\{
        v\in L^0(\Omega;\R^N)
        \st
        \rho_\varphi(\lambda v)<\infty
        \text{ for some }\lambda>0
    \right\},
\]
equipped with the Luxemburg norm
\[
    \norm{v}_\varphi
    :=
    \inf
    \left\{
        \lambda>0
        \st
        \rho_\varphi(v/\lambda)\leq1
    \right\}.
\]
The scalar space $L^\varphi(\Omega)$ and its norm are defined
analogously. Under the uniform $\Delta_2$-condition,
$L^\varphi(\Omega;\R^N)$ coincides with the finite-modular class
defined by $\rho_\varphi(v)<\infty$.

The natural space for admissible functions is
\[
    X_\varphi
    :=
    \left\{
        u\in L^2(\Omega)\cap W^{1,1}(\Omega)
        \st
        \rho_\varphi(\nabla u)<\infty
    \right\},
\]
with norm
\[
    \norm{u}_{X_\varphi}
    :=
    \norm{u}_{L^2(\Omega)}+\norm{\nabla u}_\varphi.
\]
For $f\in L^2(\Omega)$, we consider
\begin{equation}
    \mathcal{F}:X_\varphi\to\R,
    \qquad
    \mathcal{F}(u)
    :=
    \frac12\norm{u-f}_{L^2(\Omega)}^2
    +\rho_\varphi(\nabla u).
    \label{D: functional}
\end{equation}

Uniform superlinearity gives $T>0$ such that $\varphi(x,t)\geq t$
for almost every $x\in\Omega$ and all $t\geq T$. For
$v\in L^\varphi(\Omega;\R^N)$ and $\lambda>\norm{v}_\varphi$,
\[
    \int_\Omega\frac{\abs{v}}{\lambda}\,dx
    \leq
    T\abs{\Omega}+\rho_\varphi(v/\lambda)
    \leq
    T\abs{\Omega}+1.
\]
Letting $\lambda\downarrow\norm{v}_\varphi$ yields the continuous
embedding
$L^\varphi(\Omega;\R^N)\hookrightarrow L^1(\Omega;\R^N)$.
Completeness of $L^\varphi(\Omega;\R^N)$ follows componentwise from
\cite[Theorem~3.3.7]{harjulehto2019generalized}.
Convergence of $u_n$ in $L^2$ and of $\nabla u_n$ in $L^\varphi$
implies convergence of both in $L^1$. The graph of the weak gradient
is therefore closed in
$L^2(\Omega)\times L^\varphi(\Omega;\R^N)$, so $X_\varphi$ is a
Banach space. Moreover, the operator
\[
    A:X_\varphi\to L^\varphi(\Omega;\R^N),
    \qquad
    Au:=\nabla u,
\]
is continuous.


\subsection{Existence of minimizers}
\label{s:3.1}

The following compactness result is the basic tool for the direct method.

\begin{lemma}[compactness]
    \label{L: compactness}
    Let $\{u_n\}_n\subset X_\varphi$ satisfy
    \[
        \sup_{n\in\N} \mathcal{F}(u_n)
        \leq C.
    \]
    Then there exist a subsequence, not relabelled, and $u\in L^2(\Omega)\cap W^{1,1}(\Omega)$ such that
    \begin{align*}
        u_n&\to u \qquad\text{strongly in }L^1(\Omega),\\
        u_n&\rightharpoonup u \qquad\text{weakly in }L^2(\Omega),\\
        \shortintertext{and}
        \nabla u_n&\rightharpoonup\nabla u \quad\ \text{weakly in }L^1(\Omega;\R^N).
    \end{align*}
\end{lemma}
\begin{proof}
The fidelity term gives boundedness of $\{u_n\}_n$ in
$L^2(\Omega)$, and $\rho_\varphi(\nabla u_n)\leq C$ for every $n$.
By uniform superlinearity, for every $M>0$ there exists $t_M>0$
such that $\varphi(x,t)\geq Mt$ for almost every $x\in\Omega$
and all $t\geq t_M$. Hence
\[
    \int_\Omega\abs{\nabla u_n}\,dx
    \leq
    t_M\abs{\Omega}
    +\frac{1}{M}\rho_\varphi(\nabla u_n)
    \leq
    t_M\abs{\Omega}+\frac{C}{M}.
\]
Taking $M=1$ and using boundedness of $\Omega$, we obtain
boundedness of $\{u_n\}_n$ in $W^{1,1}(\Omega)$.
Rellich--Kondrachov and weak compactness in $L^2(\Omega)$ give,
after passing to a subsequence,
$u_n\to u$ strongly in $L^1(\Omega)$ and
$u_n\rightharpoonup u$ weakly in $L^2(\Omega)$.
The two limits agree by uniqueness of distributional limits.

For every $\lambda\geq t_M$, we also have
\[
    \sup_n
    \int_{\{\abs{\nabla u_n}>\lambda\}}
    \abs{\nabla u_n}\,dx
    \leq
    \frac{1}{M}\sup_n\rho_\varphi(\nabla u_n)
    \leq\frac{C}{M}.
\]
Since $M>0$ was arbitrary, the gradients are uniformly integrable.
By Dunford--Pettis \cite[Theorem~1.43]{evans2015measure},
applied componentwise, there exist a further subsequence and
$g\in L^1(\Omega;\R^N)$ such that
$\nabla u_n\rightharpoonup g$ weakly in $L^1(\Omega;\R^N)$.

Passing to the limit in the distributional identity gives
\[
    \int_\Omega u\,\partial_i\eta\,dx
    =
    -\int_\Omega g_i\eta\,dx
    \qquad
    \text{for every }\eta\in C_c^\infty(\Omega),
    \quad i\in\{1,\dots,N\}.
\]
Thus $u\in W^{1,1}(\Omega)$ and $g=\nabla u$, which proves the claim.
\end{proof}

\begin{remark}[lower semicontinuity]
    \label{R: lsc}
    The fidelity $u\mapsto\frac12\norm{u-f}_{L^2(\Omega)}^2$ is weakly
    lower semicontinuous on $L^2(\Omega)$. For the modular, let
    $v_n\to v$ in $L^1(\Omega;\R^N)$. A subsequence realizing
    $\liminf_n\rho_\varphi(v_n)$ admits a further subsequence
    converging to $v$ almost everywhere. Pointwise lower
    semicontinuity and Fatou's lemma then give
    $\rho_\varphi(v)\leq\liminf_n\rho_\varphi(v_n)$
    (see \cite[Lemma~3.1.4]{harjulehto2019generalized}).
    The epigraph of $\rho_\varphi$ in $L^1(\Omega;\R^N)\times\R$
    is convex and norm closed, hence weakly closed by separation.
    Thus $\rho_\varphi$ is also weakly lower semicontinuous on
    $L^1(\Omega;\R^N)$.
\end{remark}

\begin{proposition}[existence of minimizers]
    \label{P: existence}
    There exists a unique minimizer of $\mathcal{F}$ over $X_\varphi$.
\end{proposition}
\begin{proof}
    The direct method applies. Let $\{u_n\}_n\subset X_\varphi$ be a minimizing sequence. Since $\mathcal{F}$ is bounded from below, we may assume $\sup_n\mathcal{F}(u_n)\leq C$, and hence
    $\sup_n\norm{u_n-f}_{L^2(\Omega)}^2\leq C$ and
    $\sup_n\rho_\varphi(\nabla u_n)\leq C$.
    Thus $\{u_n\}_n$ is bounded in $L^2(\Omega)$, and \cref{L: compactness} yields a subsequence and $\overline u\in L^2(\Omega)\cap W^{1,1}(\Omega)$ such that
    $u_n\rightharpoonup\overline u$ in $L^2(\Omega)$ and
    $\nabla u_n\rightharpoonup\nabla\overline u$ in $L^1(\Omega;\R^N)$.
    By \cref{R: lsc},
    $\rho_\varphi(\nabla\overline u)
    \leq\liminf_{n\to\infty}\rho_\varphi(\nabla u_n)<\infty$,
    so $\overline u\in X_\varphi$, and again by lower semicontinuity,
    \[
        \mathcal{F}(\overline u)
        \leq\liminf_{n\to\infty}\mathcal{F}(u_n)
        =\inf_{u\in X_\varphi}\mathcal{F}(u).
    \]
    Uniqueness follows from the strict convexity of the fidelity term on $L^2(\Omega)$ and the convexity of the integral term.
\end{proof}


\subsection{Duality in Orlicz spaces}
\label{s:3.2}

For almost every $x\in\Omega$, define the pointwise conjugate
\[
    \varphi^*(x,s)
    :=
    \sup_{t\geq0}\{st-\varphi(x,t)\},
    \qquad s\in\R.
\]
The Orlicz modular uses the restriction to $s\geq0$.
The radial structure gives
\[
    \Phi^*(x,\xi)
    =
    \sup_{z\in\R^N}\{\xi\cdot z-\Phi(x,z)\}
    =
    \varphi^*(x,\abs{\xi}),
    \qquad \xi\in\R^N.
\]
Young's inequality reads
$st\leq\varphi(x,t)+\varphi^*(x,s)$ for $s,t\geq0$.

We write
\[
    \rho_{\varphi^*}(p)
    :=
    \int_\Omega\varphi^*(x,\abs{p(x)})\,dx,
\]
and define
\[
    L^{\varphi^*}(\Omega;\R^N)
    :=
    \left\{
        p\in L^0(\Omega;\R^N)
        \st
        \rho_{\varphi^*}(\lambda p)<\infty
        \text{ for some }\lambda>0
    \right\},
\]
with the corresponding Luxemburg norm $\norm{p}_{\varphi^*}$.
Here $\rho_{\varphi^*}$ is the modular generated by the pointwise
conjugate integrand, whereas $\rho_\varphi^*$ denotes the Fenchel
conjugate of the functional $\rho_\varphi$.
Since $\varphi^*$ need not satisfy a $\Delta_2$-condition, the
Orlicz space $L^{\varphi^*}$ must be distinguished from the
finite-modular class
\[
    \operatorname{dom}\rho_{\varphi^*}
    :=
    \left\{
        p\in L^{\varphi^*}(\Omega;\R^N)
        \st
        \rho_{\varphi^*}(p)<\infty
    \right\}.
\]
The conjugate functional $\rho_\varphi^*$ is finite precisely on
$\operatorname{dom}\rho_{\varphi^*}$ and equals $+\infty$ otherwise.

\begin{lemma}[duality pairing and dual identification]
    \label{L: dual space}
    Assume that $\varphi$ satisfies the standing assumptions of \cref{s:3}. Then the associate space of $L^\varphi(\Omega;\R^N)$, in the sense of \cite[Definition~3.4.5]{harjulehto2019generalized}, is $L^{\varphi^*}(\Omega;\R^N)$, with duality pairing
    \[
        \langle p,v\rangle
        :=
        \int_\Omega p\cdot v\,dx.
    \]
    Since $\varphi$ satisfies the uniform $\Delta_2$-condition, $L^\varphi(\Omega;\R^N)$ has order-continuous norm. Consequently, the norm dual can be identified with its associate space, and hence
    \[
        \bigl(L^\varphi(\Omega;\R^N)\bigr)^*
        \simeq
        L^{\varphi^*}(\Omega;\R^N)
    \]
    through the above pairing. In particular,
    \[
        \int_\Omega \abs{p\cdot v}\,dx
        \leq
        2\norm{p}_{\varphi^*}\norm{v}_{\varphi}
    \]
    for all $v\in L^\varphi(\Omega;\R^N)$ and $p\in L^{\varphi^*}(\Omega;\R^N)$.
\end{lemma}
\begin{proof}
    By \cite[Definition~3.4.5 and Theorem~3.4.6]{harjulehto2019generalized}, the associate space of
    $L^\varphi(\Omega;\mathbb R^N)$ is $L^{\varphi^*}(\Omega;\mathbb R^N)$, with comparable norms. This gives the integral pairing and the H\"older inequality.

    We first show that the Luxemburg norm is order continuous under the present assumptions. Let
    $0\leq \abs{v_k}\leq \abs{v}$ and $v_k\to0$ a.e., with
    $v\in L^\varphi(\Omega;\mathbb R^N)$. Since $v\in L^\varphi$, there exists
    $\lambda_0>0$ such that $\rho_\varphi(\lambda_0v)<\infty$.
    By the uniform $\Delta_2$-condition, this implies
    $\rho_\varphi(\lambda v)<\infty$ for every $\lambda>0$.
    For fixed $\lambda>0$, we have
    \[
        0\leq
        \varphi(x,\lambda\abs{v_k(x)})
        \leq
        \varphi(x,\lambda\abs{v(x)})
        \in L^1(\Omega),
    \]
    and the integrand converges pointwise to zero. By dominated convergence,
    $\rho_\varphi(\lambda v_k)\to0$.
    Taking $\lambda=1/\varepsilon$, we find that for all sufficiently large $k$,
    $\rho_\varphi(v_k/\varepsilon)\leq1$. Thus
    $\norm{v_k}_\varphi\leq\varepsilon$. Since $\varepsilon>0$ was arbitrary,
    $\norm{v_k}_\varphi\to0$. This proves order continuity.

    It remains to exclude singular functionals. Let
    $\ell\in (L^\varphi(\Omega;\mathbb R^N))^*$. For $i=1,\dots,N$, define
    $\nu_i(E):=\ell(\chi_E e_i)$ for measurable $E\subset\Omega$.
    This is well-defined because $\varphi(\cdot,1)\in L^1(\Omega)$, and therefore
    $\chi_E e_i\in L^\varphi(\Omega;\mathbb R^N)$. The linearity of $\ell$
    implies finite additivity.

    We next prove countable additivity. Let $(E_j)_{j\in\mathbb N}$ be pairwise
    disjoint and set
    $E:=\bigcup_{j=1}^\infty E_j$ and
    $E^{(m)}:=\bigcup_{j=1}^m E_j$.
    Then $\chi_{E^{(m)}}e_i\to\chi_Ee_i$ a.e. and
    $\abs{\chi_Ee_i-\chi_{E^{(m)}}e_i}
    =\chi_{E\setminus E^{(m)}}$.
    Since $E\setminus E^{(m)}\downarrow\emptyset$, order continuity of the
    Luxemburg norm gives
    $\norm{\chi_Ee_i-\chi_{E^{(m)}}e_i}_\varphi\to0$.
    By the continuity of $\ell$,
    $\nu_i(E)-\nu_i(E^{(m)})
    =\ell((\chi_E-\chi_{E^{(m)}})e_i)\to0$.
    Using finite additivity for $E^{(m)}$, we obtain
    \[
        \nu_i(E)
        =
        \lim_{m\to\infty}\sum_{j=1}^m\nu_i(E_j)
        =
        \sum_{j=1}^\infty\nu_i(E_j).
    \]
    Thus $\nu_i$ is countably additive.

    Moreover, if $\abs{E}=0$, then $\chi_Ee_i=0$ in
    $L^\varphi(\Omega;\mathbb R^N)$, and hence $\nu_i(E)=0$.
    Therefore $\nu_i\ll\mathcal L^N$.

    Finally, we show finite variation explicitly. For every finite measurable
    partition $(E_j)_{j=1}^m$ of $\Omega$, choose signs
    $\varepsilon_j\in\{-1,1\}$ such that
    $\varepsilon_j\nu_i(E_j)=\abs{\nu_i(E_j)}$.
    Then
    \[
        \sum_{j=1}^m\abs{\nu_i(E_j)}
        =
        \ell\left(\sum_{j=1}^m\varepsilon_j\chi_{E_j}e_i\right)
        \leq
        \norm{\ell}
        \left\|
        \sum_{j=1}^m\varepsilon_j\chi_{E_j}e_i
        \right\|_\varphi.
    \]
    Since
    $\left|\sum_{j=1}^m\varepsilon_j\chi_{E_j}e_i\right|=\chi_\Omega$,
    the lattice property of the Luxemburg norm gives
    \[
        \left\|
        \sum_{j=1}^m\varepsilon_j\chi_{E_j}e_i
        \right\|_\varphi
        =
        \norm{\chi_\Omega e_i}_\varphi
        <
        \infty.
    \]
    Hence
    \[
        \abs{\nu_i}(\Omega)
        \leq
        \norm{\ell}\,\norm{\chi_\Omega e_i}_\varphi
        <
        \infty.
    \]
    Thus $\nu_i$ is a finite signed measure. Since
    $\nu_i\ll\mathcal L^N$, the Radon--Nikodym theorem gives
    $p_i\in L^1(\Omega)$ such that
    $\nu_i(E)=\int_E p_i\,dx$ for every measurable $E\subset\Omega$.

        For $p=(p_1,\dots,p_N)$, the identity
    $\ell(v)=\int_\Omega p\cdot v\,dx$ holds for all simple vector
    fields $v$. It extends to bounded fields by uniform approximation
    with simple fields. Indeed, uniform convergence implies
    convergence in the Luxemburg norm by monotonicity and
    $\chi_\Omega e_1\in L^\varphi(\Omega;\R^N)$, and it also implies
    convergence of the integrals since $p\in L^1(\Omega;\R^N)$.

    Given $v\in L^\varphi(\Omega;\R^N)$, test this representation with
    $w_m=\min\{\abs{v},m\}p/\abs{p}$, setting $w_m=0$ where $p=0$.
    Since $\abs{w_m}\leq\abs{v}$, monotone convergence gives
    \[
        \int_\Omega \abs{p}\abs{v}\,dx
        =\lim_{m\to\infty}\ell(w_m)
        \leq \norm{\ell}\,\norm{v}_\varphi.
    \]
    Taking $v=g e_1$ for scalar $g\in L^\varphi(\Omega)$ and applying
    \cite[Theorem~3.4.6]{harjulehto2019generalized} gives
    $p\in L^{\varphi^*}(\Omega;\R^N)$. Bounded truncations of
    $v$ converge to $v$ in the Luxemburg norm by order continuity.
    Passing to the limit by continuity of $\ell$ and the generalized
    H\"older inequality
    \cite[Lemma~3.2.11]{harjulehto2019generalized} gives the integral
    representation for every $v\in L^\varphi(\Omega;\R^N)$.
\end{proof}

By definition, the Fenchel conjugate of $\rho_\varphi$ on the associated space is
\[
    \rho_\varphi^*(p):=\sup_{v\in L^\varphi(\Omega;\R^N)}
    \left\{\int_\Omega p\cdot v\,dx-\rho_\varphi(v)\right\}.
\]
Just as in Lebesgue spaces, this conjugate can be expressed pointwise almost everywhere via the conjugate of the modular.
\begin{proposition}[Fenchel conjugate of the modular]
    \label{P: dual modular}
    For $p\in L^{\varphi^*}(\Omega;\R^N)$,
    \[
        \rho_\varphi^*(p)
        =\int_\Omega\Phi^*(x,p(x))\,dx
        =\int_\Omega\varphi^*(x,\abs{p(x)})\,dx.
    \]
    In particular, $\rho_\varphi^*(p)<\infty$ if and only if $p\in\operatorname{dom}\rho_{\varphi^*}$.
\end{proposition}
\begin{proof}
    The following argument is a direct radial specialization of the standard conjugacy formula for convex integral functionals with normal integrands in~\cite[Theorem~2]{rockafellar1968integrals}.

    We first compute the pointwise conjugate. For a.e. $x\in\Omega$ and every
    $\xi\in\R^N$, we have
    \[
        \Phi^*(x,\xi)
        =
        \sup_{z\in\R^N}
        \{\xi\cdot z-\varphi(x,|z|)\}
        =
        \sup_{t\ge 0}
        \{\abs{\xi}t-\varphi(x,t)\}
        =
        \varphi^*(x,\abs{\xi}).
    \]
    Indeed, the inequality ``$\le$'' follows from
    $\xi\cdot z\le \abs{\xi}\,\abs{z}$, while the reverse inequality follows,
    for $\xi\neq0$, by choosing $z=t\xi/\abs{\xi}$. The case $\xi=0$ is immediate.

    Let now $p\in L^{\varphi^*}(\Omega;\R^N)$. For every
    $v\in L^\varphi(\Omega;\R^N)$, the pointwise Fenchel inequality gives
    \[
        p(x)\cdot v(x)-\Phi(x,v(x))
        \le
        \Phi^*(x,p(x))
        =
        \varphi^*(x,\abs{p(x)})
        \qquad\text{for a.e. }x\in\Omega.
    \]
    After integration and taking the supremum over $v$, we obtain
    \[
        \rho_\varphi^*(p)
        \le
        \int_\Omega \varphi^*(x,\abs{p})\,dx.
    \]

    It remains to prove the reverse inequality. For $R>0$ and $m\in\N$, define
    \[
        \begin{aligned}
        M_R(x)
        &:=
        \sup_{0\le t\le R}
        \left\{
            \abs{p(x)}t-\varphi(x,t)
        \right\},\\
        D_{R,m}
        &:=
        \left\{
            \frac{jR}{2^m}: j=0,\dots,2^m
        \right\},\\
        M_{R,m}(x)
        &:=
        \max_{t\in D_{R,m}}
        \left\{
            \abs{p(x)}t-\varphi(x,t)
        \right\}.
        \end{aligned}
    \]
    Since $D_{R,m}$ is finite and $x\mapsto\varphi(x,t)$ is measurable for every
    fixed $t\ge0$, the function $M_{R,m}$ is measurable. Moreover,
    $0\in D_{R,m}$, and hence $M_{R,m}(x)\ge0$.

    Choose $t_{R,m}(x)\in D_{R,m}$ as the smallest grid point where the maximum
    defining $M_{R,m}(x)$ is attained. This gives a measurable function
    $t_{R,m}$. Indeed, if $D_{R,m}=\{s_0,\dots,s_{2^m}\}$, then the sets on
    which $t_{R,m}=s_j$ are obtained by finite intersections and unions of
    measurable sets of the form
    \[
        \left\{
            \abs{p}s_j-\varphi(\cdot,s_j)
            \ge
            \abs{p}s_k-\varphi(\cdot,s_k)
        \right\}.
    \]

    Define
    \[
        e_p(x)
        :=
        \begin{cases}
            \dfrac{p(x)}{\abs{p(x)}}, & p(x)\neq0,\\[1ex]
            e_1, & p(x)=0,
        \end{cases}
    \]
    where $e_1$ is the first coordinate vector in $\R^N$, and set
    $v_{R,m}(x):=t_{R,m}(x)e_p(x)$. Then $v_{R,m}$ is measurable and satisfies
    $\abs{v_{R,m}(x)}=t_{R,m}(x)\le R$. Therefore
    \[
        \rho_\varphi(v_{R,m})
        =
        \int_\Omega \varphi(x,\abs{v_{R,m}})\,dx
        \le
        \int_\Omega \varphi(x,R)\,dx
        <\infty,
    \]
    by the standing assumption that $\varphi(\cdot,R)\in L^1(\Omega)$. Hence
    $v_{R,m}\in L^\varphi(\Omega;\R^N)$. For this admissible test field, we have
    pointwise
    \[
        p(x)\cdot v_{R,m}(x)-\Phi(x,v_{R,m}(x))
        =
        \abs{p(x)}t_{R,m}(x)-\varphi(x,t_{R,m}(x))
        =
        M_{R,m}(x).
    \]
    Thus, by the definition of the conjugate functional,
    \[
        \rho_\varphi^*(p)
        \ge
        \int_\Omega M_{R,m}(x)\,dx.
    \]

    For fixed $R>0$, the grids $D_{R,m}$ are increasing in $m$ and their union
    is dense in $[0,R]$. Since
    $t\mapsto \abs{p(x)}t-\varphi(x,t)$ is continuous on $[0,R]$ for a.e. $x$,
    we have $M_{R,m}(x)\uparrow M_R(x)$ as $m\to\infty$. By the monotone
    convergence theorem,
    \[
        \rho_\varphi^*(p)
        \ge
        \int_\Omega M_R(x)\,dx.
    \]
    Finally,
    \[
        M_R(x)\uparrow
        \sup_{t\ge0}
        \left\{
            \abs{p(x)}t-\varphi(x,t)
        \right\}
        =
        \varphi^*(x,\abs{p(x)})
        \qquad\text{as }R\to\infty.
    \]
    A second application of the monotone convergence theorem yields
    \[
        \rho_\varphi^*(p)
        \ge
        \int_\Omega \varphi^*(x,\abs{p})\,dx.
    \]
    Combining both inequalities gives
    \[
        \rho_\varphi^*(p)
        =
        \int_\Omega \varphi^*(x,\abs{p})\,dx
        =
        \int_\Omega \Phi^*(x,p(x))\,dx.
    \]
    In particular, this identity is understood in the extended-valued sense: the
    conjugate is finite precisely on the finite-modular class
    \[
        \operatorname{dom}\rho_{\varphi^*}
        =
        \left\{
            p\in L^{\varphi^*}(\Omega;\R^N):
            \int_\Omega \varphi^*(x,\abs{p})\,dx<\infty
        \right\}.
        \qedhere
    \]
\end{proof}

\begin{corollary}[subdifferential of the modular]
    \label{C: subdiff representation}
    For $v\in L^\varphi(\Omega;\R^N)$,
    \[
        \partial\rho_\varphi(v)
        =\left\{p\in L^{\varphi^*}(\Omega;\R^N)\st
        p(x)\in\partial_z\Phi(x,v(x))\ \text{for a.e. }x\in\Omega\right\}.
    \]
\end{corollary}
\begin{proof}
    By the Fenchel equality for the modular, $p\in\partial\rho_\varphi(v)$ is equivalent to
    $\rho_\varphi(v)+\rho_\varphi^*(p)=\int_\Omega p\cdot v\,dx$.
    Using \cref{P: dual modular}, this identity can be written as
    \[
        \int_\Omega
        \left(
            \Phi(x,v(x))+\Phi^*(x,p(x))-p(x)\cdot v(x)
        \right)\,dx=0.
    \]
    The integrand is nonnegative by the pointwise Fenchel inequality. Hence the integral vanishes if and only if
    $\Phi(x,v(x))+\Phi^*(x,p(x))=p(x)\cdot v(x)$
    for almost every $x\in\Omega$. This pointwise Fenchel equality is equivalent to
    $p(x)\in\partial_z\Phi(x,v(x))$ for a.e. $x\in\Omega$,
    which gives the stated characterization. The converse follows by reversing the same argument.
\end{proof}


\subsection{Optimality conditions}
\label{s:3.3}

Let $i:X_\varphi\hookrightarrow L^2(\Omega)$ denote the canonical
inclusion. Since $\varphi(\cdot,t)\in L^1(\Omega)$ for every fixed
$t\geq0$, we have $C_c^\infty(\Omega)\subset X_\varphi$, so $i$ has
dense range. With $Y:=L^\varphi(\Omega;\R^N)$ and the continuous
linear operator $A:X_\varphi\to Y$, $Au:=\nabla u$, we write
\[
    \mathcal{F}=F_0+G\circ A,
    \qquad
    F_0(u):=\frac12\norm{iu-f}_{L^2(\Omega)}^2,
    \qquad
    G(v):=\rho_\varphi(v).
\]
The fidelity $F_0$ is continuous and convex. The modular $G$ is
convex and lower semicontinuous on $Y$ by the embedding into $L^1$
and \cref{R: lsc}. The uniform $\Delta_2$-condition gives
$\operatorname{dom}G=Y$, so $u_0=0$ satisfies
$Au_0=0\in\operatorname{int}(\operatorname{dom}G)$.
Since a primal minimizer exists by \cref{P: existence}, the
Fenchel--Rockafellar theorem \cite[Theorem~5.12]{NAO:2020}
yields strong duality and dual attainment.

\begin{lemma}[dual problem]
    \label{L: dual problem for Orlicz}
    The Fenchel dual problem corresponding to minimization of $\mathcal{F}=F_0+G\circ A$ is
    \[
        \sup_{p\in K_\varphi}
        \left\{
        -\frac12\norm{\operatorname{div}p}_{L^2(\Omega)}^2
        -\int_\Omega f\,\operatorname{div}p\,dx
        -\int_\Omega\varphi^*(x,\abs{p(x)})\,dx
        \right\},
    \]
    where
    \[
    K_\varphi
    :=
    \left\{
        p\in L^{\varphi^*}(\Omega;\R^N)\st
        \begin{aligned}
            &\operatorname{div}p\in L^2(\Omega),\\
            &\int_\Omega p\cdot\nabla u\,dx
            =
            -\int_\Omega (\operatorname{div}p)u\,dx
            \quad\text{for all }u\in X_\varphi
        \end{aligned}
    \right\}.
\]
    The integration-by-parts identity is the weak formulation of the homogeneous normal boundary condition $p\cdot\nu=0$ on $\partial\Omega$.
    This problem admits a maximizer.
\end{lemma}
\begin{proof}
        By \cref{L: dual space}, we identify $Y^*$ with
    $L^{\varphi^*}(\Omega;\R^N)$. For $p$ in this space,
    $\langle A^*p,u\rangle=\int_\Omega p\cdot\nabla u\,dx$. Since bounded
    fields lie in $L^\varphi$, H\"older's inequality gives $p\in L^1$.

    Let $\ell\in X_\varphi^*$ with $F_0^*(\ell)<\infty$. Testing
    the conjugate on $tu$, $t\in\R$, and maximizing in $t$ gives,
    for every $u\neq0$,
    \[
        F_0^*(\ell)+\frac12\norm{f}_{L^2(\Omega)}^2
        \geq
        \frac{\abs{\ell(u)+(f,iu)_{L^2(\Omega)}}^2}
        {2\norm{iu}_{L^2(\Omega)}^2}.
    \]
    Hence $\ell$ is bounded with respect to the $L^2$ norm.
    Since $i$ has dense range, the Riesz representation theorem
    gives $\ell=i^*g$ for a unique $g\in L^2(\Omega)$.
    Conversely, for $\ell=i^*g$, density allows us to take the
    supremum over $L^2(\Omega)$ and gives
    \[
        F_0^*(i^*g)
        =
        \sup_{w\in L^2(\Omega)}
        \left\{(g,w)_{L^2(\Omega)}
        -\frac12\norm{w-f}_{L^2(\Omega)}^2\right\}
        =
        \frac12\norm{g}_{L^2(\Omega)}^2+(g,f)_{L^2(\Omega)}.
    \]

    For $\ell=-A^*p$, the representation $\ell=i^*g$ is equivalent to
    \[
        \int_\Omega p\cdot\nabla u\,dx
        =-\int_\Omega g\,iu\,dx
        \qquad\text{for all }u\in X_\varphi.
    \]
    Testing on $C_c^\infty(\Omega)$ identifies $g=\operatorname{div}p$
    in distributions. The identity on $X_\varphi$ is then precisely
    the weak boundary condition in the definition of $K_\varphi$.
    Thus $F_0^*(-A^*p)$ is finite exactly for $p\in K_\varphi$,
    with the value above for $g=\operatorname{div}p$.

    \Cref{P: dual modular} gives
    $G^*(p)=\int_\Omega\varphi^*(x,\abs{p(x)})\,dx$, with value $+\infty$
    when $p\notin\operatorname{dom}\rho_{\varphi^*}$. Combining the two
    conjugates proves the dual formula.
\end{proof}

\begin{proposition}[optimality conditions]
    The dual problem in \cref{L: dual problem for Orlicz} admits a maximizer $p\in K_\varphi$. Let $u\in X_\varphi$ be the primal minimizer and let $p\in K_\varphi$ be a dual maximizer. Then
    \[
        \operatorname{div}p=u-f
        \qquad\text{in }L^2(\Omega),
    \]
    and
    \[
        p(x)\in\partial_z\Phi(x,\nabla u(x))
        \qquad\text{for a.e. }x\in\Omega.
    \]
    If $\nabla u(x)\neq0$, this can be written as
    \begin{equation}
        \label{P: optimality conditions Orlicz}
        p(x)
        =
        a(x)\frac{\nabla u(x)}{\abs{\nabla u(x)}},
        \qquad
        a(x)\in\partial_t\varphi(x,\abs{\nabla u(x)}).
    \end{equation}
\end{proposition}

\begin{proof}
Applying the Fenchel--Rockafellar theorem gives
$-A^*p\in\partial F_0(u)$ and $p\in\partial G(\nabla u)$.
The fidelity is differentiable on $X_\varphi$, with
$\partial F_0(u)=\{i^*(iu-f)\}$. The first condition therefore
gives $-A^*p=i^*(iu-f)$, which identifies
$\operatorname{div}p=iu-f$ in $L^2(\Omega)$ by the argument in
\cref{L: dual problem for Orlicz}.
By \cref{C: subdiff representation}, the second condition is precisely
$p(x)\in\partial_z\Phi(x,\nabla u(x))$ for almost every $x\in\Omega$.
The radial description follows from the subdifferential of
$z\mapsto\varphi(x,\abs{z})$.
\end{proof}


\section{The scaled logarithmic model}
\label{s:4}

We now specialize the framework from \cref{s:3} to a logarithmically
corrected total-variation density. The local scale $\kappa$ sets the
gradient threshold for the correction. For each fixed $\gamma>0$ the
energy is posed on an Orlicz--Sobolev space. We derive the explicit dual
formulation and proximal map, then show that the energies
$\Gamma$-converge to the $BV$-based ROF functional as $\gamma\downarrow0$.

\subsection{Definition and basic properties}
\label{s:4.1}

Let $\alpha>0$ be fixed. We prescribe a local gradient scale
\begin{equation}
    \kappa\in L^\infty(\Omega),
    \qquad
    0<\kappa_-\leq \kappa(x)\leq \kappa_+<\infty
    \quad\text{for a.e. }x\in\Omega.
    \label{D: kappa}
\end{equation}
For $\gamma>0$, define the scaled logarithmic integrand
\begin{equation}
    \varphi_\gamma(x,t)
    :=
    \alpha t\left(1+\gamma\log^+\frac{t}{\kappa(x)}\right),
    \qquad t\geq0.
    \label{D: varphi gamma}
\end{equation}
The scale function $\kappa$ remains fixed throughout this section and is therefore suppressed in the notation. Equivalently,
\begin{equation}
    \label{D: varphi gamma piecewise}
    \varphi_\gamma(x,t)
    =
    \begin{cases}
        \alpha t,
        & 0\leq t\leq \kappa(x),\\[0.5ex]
        \alpha t\left(1+\gamma\log\dfrac{t}{\kappa(x)}\right),
        & t\geq \kappa(x).
    \end{cases}
\end{equation}

\begin{remark}[smooth logarithmic variant]
As a smooth alternative to \eqref{D: varphi gamma piecewise}, one may consider
\[
\widetilde{\varphi}_{\gamma}(x,t):=\alpha t\left(1+\gamma\log\left(1+\frac{t}{\kappa(x)}\right)\right),\qquad t\geq0.
\]
Both $\varphi_\gamma$ and $\widetilde{\varphi}_{\gamma}$ generate the same generalized Orlicz and Orlicz--Sobolev spaces with equivalent Luxemburg norms. However, the pointwise dual proximal map related to $\widetilde{\varphi}_{\gamma}$ requires the inversion of a scalar nonlinear function rather than admitting the explicit piecewise formula derived in \cref{s:4.3}.
\end{remark}

For a nonnegative function $g$ on $(0,\infty)$, we use
$(\mathrm{Inc})_p$ and $(\mathrm{Dec})_q$ to mean that
$t\mapsto g(t)/t^p$ is increasing and
$t\mapsto g(t)/t^q$ is decreasing, respectively.
The condition $(\mathrm{aInc})_p$ means that there exists $L\geq1$
such that $g(s)/s^p\leq Lg(t)/t^p$ whenever $0<s<t$.
For functions depending on $x$, uniformity means that the
exponents and comparison constants can be chosen independently
of $x$.

\begin{lemma}[properties of the scaled logarithmic integrand]
    \label{L: varphi gamma properties}
    For every $\gamma>0$, the function
    \[
        \varphi_\gamma:\Omega\times[0,\infty)\to[0,\infty)
    \]
    satisfies the standing assumptions of \cref{s:3}. More precisely:
    \begin{enumerate}
        \item For almost every $x\in\Omega$, the map $t\mapsto\varphi_\gamma(x,t)$ is convex, continuous, increasing, satisfies $\varphi_\gamma(x,0)=0$, and
            \[
                \lim_{t\to\infty}
                \frac{\varphi_\gamma(x,t)}{t}
                =\infty.
            \]

        \item The superlinear growth is uniform in $x$:
            \[
                \lim_{t\to\infty}
                \operatorname*{ess\,inf}_{x\in\Omega}
                \frac{\varphi_\gamma(x,t)}{t}
                =\infty.
            \]

        \item The uniform $\Delta_2$-condition holds:
            \[
                \varphi_\gamma(x,2t)
                \leq
                2(1+\gamma\log2)\varphi_\gamma(x,t)
                \qquad\text{for a.e. }x\in\Omega,
                \quad t\geq0.
            \]

        \item For almost every $x\in\Omega$, the function $t\mapsto\varphi_\gamma(x,t)$ satisfies $(\mathrm{Inc})_1$, but it does not satisfy $(\mathrm{aInc})_p$ for any $p>1$.

        \item For every $q\geq1+\gamma$, the function $t\mapsto\varphi_\gamma(x,t)$ satisfies $(\mathrm{Dec})_q$ uniformly in $x$.

        \item As $\gamma\downarrow0$,
            \[
                \varphi_\gamma(x,t)\to \alpha t
            \]
            locally uniformly in $t\in[0,\infty)$, uniformly in $x$.
    \end{enumerate}
\end{lemma}
\begin{proof}
Since $\kappa$ is measurable, $x\mapsto\varphi_\gamma(x,t)$ is measurable for every $t\geq0$; since $\kappa\geq\kappa_->0$ and $\Omega$ is bounded, $\varphi_\gamma(\cdot,t)\in L^1(\Omega)$ for every $t\geq0$. The piecewise representation \eqref{D: varphi gamma piecewise} shows continuity, monotonicity, and $\varphi_\gamma(x,0)=0$. For $0<t<\kappa(x)$, the derivative with respect to $t$ equals $\alpha$, while for $t>\kappa(x)$ it equals
$\alpha\left(1+\gamma+\gamma\log\frac{t}{\kappa(x)}\right)$,
which is increasing in $t$. Hence $t\mapsto\varphi_\gamma(x,t)$ is convex.

The lower bound $\varphi_\gamma(x,t)\geq \alpha t$ is immediate from $\log^+\geq0$. Moreover,
\[
    \frac{\varphi_\gamma(x,t)}{t}
    \geq
    \alpha\left(1+\gamma\log\frac{t}{\kappa_+}\right)
    \qquad\text{for }t\geq\kappa_+,
\]
which proves uniform superlinear growth.

To prove the $\Delta_2$-condition, note that $\log^+(2s)\leq\log2+\log^+s$ for $s\geq0$. Therefore
\[
    \varphi_\gamma(x,2t)
    \leq
    2\alpha t\left(1+\gamma\log2+\gamma\log^+\frac{t}{\kappa(x)}\right)
    \leq
    2(1+\gamma\log2)\varphi_\gamma(x,t).
\]
The $(\mathrm{Inc})_1$ property follows from the monotonicity of
$\varphi_\gamma(x,t)/t
=\alpha\left(1+\gamma\log^+\frac{t}{\kappa(x)}\right)$.

Let $p>1$. Suppose that $\varphi_\gamma(x,\cdot)$ satisfied $(\mathrm{aInc})_p$ with constant $L$. Then for $0<s<t<\kappa(x)$,
\[
    \frac{\varphi_\gamma(x,s)}{s^p}
    \leq
    L\frac{\varphi_\gamma(x,t)}{t^p},
    \qquad\text{that is,}\qquad
    \alpha s^{1-p}
    \leq
    L\alpha t^{1-p},
\]
since $\varphi_\gamma(x,r)=\alpha r$ for $0<r<\kappa(x)$. Fixing $t\in(0,\kappa(x))$ and letting $s\downarrow0$ gives a contradiction. Hence $(\mathrm{aInc})_p$ fails for every $p>1$.

For the $(\mathrm{Dec})_q$ property, set
$h_q(t):=\varphi_\gamma(x,t)/t^q$.
On $(0,\kappa(x))$, this is $\alpha t^{1-q}$ and is decreasing for $q\geq1$. On $(\kappa(x),\infty)$,
\[
    h_q'(t)
    =
    \alpha t^{-q}
    \left(
        \gamma-(q-1)\left(1+\gamma\log\frac{t}{\kappa(x)}\right)
    \right),
\]
which is nonpositive for $q\geq1+\gamma$. Thus $(\mathrm{Dec})_q$ holds uniformly in $x$. Finally, on each interval $0\leq t\leq R$,
\[
    0\leq \varphi_\gamma(x,t)-\alpha t
    \leq
    \alpha\gamma R\log^+\frac{R}{\kappa_-},
\]
which gives local uniform convergence to $\alpha t$.
\end{proof}

For $\gamma>0$, set
\[
    X_\gamma
    :=
    \left\{u\in L^2(\Omega)\cap W^{1,1}(\Omega)\st
    \int_\Omega\varphi_\gamma(x,\abs{\nabla u})\,dx<\infty
    \right\}.
\]
We consider
\begin{equation}
    \mathcal{F}_\gamma(u)
    :=
    \frac12\norm{u-f}_{L^2(\Omega)}^2
    +
    \int_\Omega
    \varphi_\gamma(x,\abs{\nabla u})\,dx,
    \qquad u\in X_\gamma.
    \label{D: scaled logarithmic problem}
\end{equation}

\begin{proposition}[existence and uniqueness]
\label{P: existence logarithmic}
For every $\gamma>0$, the functional $\mathcal{F}_\gamma$ admits a unique minimizer in $X_\gamma$.
\end{proposition}
\begin{proof}
By \cref{L: varphi gamma properties}, the integrand $\varphi_\gamma$ satisfies the standing assumptions of \cref{s:3}. Hence \cref{P: existence} yields existence of a minimizer in $X_\gamma$. Uniqueness follows from the strict convexity of the $L^2$-fidelity term and the convexity of the regularization term.
\end{proof}

\subsection{Fenchel duality and optimality system}
\label{s:4.2}

\begin{lemma}[dual problem for the scaled logarithmic regularizer]
\label{L: dual problem logarithmic}
The Fenchel dual problem corresponding to \eqref{D: scaled logarithmic problem} is
\[
    \sup_{p\in K_\gamma}
    \left
    \{
    -\frac12\norm{\operatorname{div}p}_{L^2(\Omega)}^2
    -\int_\Omega f\,\operatorname{div}p\,dx
    -\int_\Omega \varphi_\gamma^*(x,\abs{p(x)})\,dx
    \right\},
\]
where
\[
    K_\gamma
    :=
    \left\{
        p\in L^{\varphi_\gamma^*}(\Omega;\R^N)\st
        \begin{aligned}
            &\operatorname{div}p\in L^2(\Omega),\\
            &\int_\Omega p\cdot\nabla u\,dx
            =
            -\int_\Omega (\operatorname{div}p)u\,dx
            \quad\text{for all }u\in X_\gamma
        \end{aligned}
    \right\}.
\]
The integration-by-parts identity is the weak homogeneous normal boundary condition.
\end{lemma}

\begin{proof}
Apply \cref{L: dual problem for Orlicz} with $\varphi=\varphi_\gamma$. The assumptions of \cref{s:3} hold by \cref{L: varphi gamma properties}.
\end{proof}

\begin{lemma}[scalar subdifferentials and conjugate]
    \label{L: logarithmic scalar subdifferentials}
    For almost every $x\in\Omega$, regarding $t\mapsto\varphi_\gamma(x,t)$ as a convex function on $[0,\infty)$, there holds
    \[
        \partial_t\varphi_\gamma(x,t)
        =
        \begin{cases}
            (-\infty,\alpha], & t=0,\\[0.4em]
            \{\alpha\}, & 0<t<\kappa(x),\\[0.4em]
            [\alpha,\alpha(1+\gamma)], & t=\kappa(x),\\[0.4em]
            \left\{\alpha\left(1+\gamma+\gamma\log\dfrac{t}{\kappa(x)}\right)\right\},
            & t>\kappa(x).
        \end{cases}
    \]
        The scalar conjugate on $\R$ is given by
    \[
        \varphi_\gamma^*(x,s)
        =
        \begin{cases}
            0,
            & s\leq\alpha,\\[0.4em]
            \kappa(x)(s-\alpha),
            & \alpha<s\leq\alpha(1+\gamma),\\[0.4em]
            \alpha\gamma\kappa(x)
            \exp\!\left(\dfrac{s/\alpha-1-\gamma}{\gamma}\right),
            & s>\alpha(1+\gamma).
        \end{cases}
    \]
    The formula is continuous at $s=\alpha$ and $s=\alpha(1+\gamma)$. Moreover,
    \begin{equation}
        \partial_s\varphi_\gamma^*(x,s)
        =
        \begin{cases}
            \{0\},
            & s<\alpha,\\[0.4em]
            [0,\kappa(x)],
            & s=\alpha,\\[0.4em]
            \{\kappa(x)\},
            & \alpha<s\leq\alpha(1+\gamma),\\[0.4em]
            \left\{
                \kappa(x)
                \exp\!\left(\dfrac{s/\alpha-1-\gamma}{\gamma}\right)
            \right\},
            & s>\alpha(1+\gamma).
        \end{cases}
        \label{D: explicit dual subgradient phi}
    \end{equation}
    Finally, for
    \[
        \Phi_\gamma(x,z):=\varphi_\gamma(x,\abs{z}),
    \]
    the corresponding vectorial subdifferential is radial. For $z\neq0$,
    \[
        \partial_z\Phi_\gamma(x,z)
        =
        \partial_t\varphi_\gamma(x,\abs{z})
        \frac{z}{\abs{z}},
    \]
    and hence
    \[
        \partial_z\Phi_\gamma(x,z)
        =
        \begin{cases}
            \left\{\alpha\dfrac{z}{\abs{z}}\right\},
            & 0<\abs{z}<\kappa(x),\\[1em]
            \left\{\lambda\dfrac{z}{\abs{z}}:\lambda\in[\alpha,\alpha(1+\gamma)]\right\},
            & \abs{z}=\kappa(x),\\[1em]
            \left\{
                \alpha\left(1+\gamma+\gamma\log\dfrac{\abs{z}}{\kappa(x)}\right)
                \dfrac{z}{\abs{z}}
            \right\},
            & \abs{z}>\kappa(x).
        \end{cases}
    \]
    At the origin,
    \[
        \partial_z\Phi_\gamma(x,0)
        =
        \left\{\xi\in\R^N:\abs{\xi}\leq\alpha\right\}.
    \]
\end{lemma}

\begin{proof}
Fix $x$ outside a null set and write $\kappa=\kappa(x)$. Then
\[
\varphi_\gamma(x,t)
=
\begin{cases}
\alpha t, & 0\leq t\leq\kappa,\\[0.4em]
\alpha t\left(1+\gamma\log\dfrac{t}{\kappa}\right),
& t\geq\kappa.
\end{cases}
\]
We first compute the scalar subdifferential of
$\varphi_\gamma(x,\cdot)$. For $0<t<\kappa$,
$\varphi_\gamma'(x,t)=\alpha$, while for $t>\kappa$,
\[
    \varphi_\gamma'(x,t)
    =
    \alpha\left(1+\gamma+\gamma\log\frac{t}{\kappa}\right).
\]
At $t=\kappa$, the left derivative is $\alpha$ and the right derivative
is $\alpha(1+\gamma)$, hence
$\partial_t\varphi_\gamma(x,\kappa)
=[\alpha,\alpha(1+\gamma)]$.
At $t=0$, the subgradient inequality over the domain $[0,\infty)$ is
\[
    \xi t\leq\varphi_\gamma(x,t)
    \qquad\text{for all }t\geq0.
\]
If $\xi\leq\alpha$, then
$\xi t\leq\alpha t\leq\varphi_\gamma(x,t)$ for every $t\geq0$.
If $\xi>\alpha$, the inequality fails for every sufficiently small
positive $t\leq\kappa$. Thus it holds precisely when $\xi\leq\alpha$,
and $\partial_t\varphi_\gamma(x,0)=(-\infty,\alpha]$.

We now compute the conjugate. On the linear branch
$0\leq t\leq\kappa$,
$st-\varphi_\gamma(x,t)=(s-\alpha)t$.
Therefore the supremum over this interval is $0$ if $s\leq\alpha$, and
$\kappa(s-\alpha)$ if $s>\alpha$.

On the logarithmic branch $t\geq\kappa$, set
\[
    \Psi_s(t)
    :=
    st-\alpha t\left(1+\gamma\log\frac{t}{\kappa}\right).
\]
Then
\[
    \Psi_s'(t)
    =
    s-\alpha(1+\gamma)-\alpha\gamma\log\frac{t}{\kappa},
    \qquad
    \Psi_s''(t)
    =
    -\frac{\alpha\gamma}{t}<0.
\]
Thus a critical point is the unique maximizer on the logarithmic branch.
Solving $\Psi_s'(t)=0$ gives
\[
    t_s
    =
    \kappa
    \exp\!\left(\frac{s/\alpha-1-\gamma}{\gamma}\right).
\]
This satisfies $t_s\geq\kappa$ exactly when
$s\geq\alpha(1+\gamma)$. If $s\leq\alpha(1+\gamma)$, then
$\Psi_s'(\kappa)\leq0$, so concavity implies that the logarithmic branch
is maximized at the endpoint $t=\kappa$, giving the same value
$\kappa(s-\alpha)$ as the right endpoint of the linear branch. Combining
this with the linear-branch computation gives the first two cases of the
conjugate.

For $s>\alpha(1+\gamma)$, the unique maximizer is $t_s>\kappa$. The
stationarity identity can be rewritten as
\[
    s-\alpha-\alpha\gamma\log\frac{t_s}{\kappa}
    =
    \alpha\gamma.
\]
Consequently,
\[
\begin{aligned}
    \varphi_\gamma^*(x,s)
    &=
    s t_s-\alpha t_s\left(1+\gamma\log\frac{t_s}{\kappa}\right)
    =
    t_s\left(s-\alpha-\alpha\gamma\log\frac{t_s}{\kappa}\right)\\
    &=
    \alpha\gamma t_s
    =
    \alpha\gamma\kappa
    \exp\!\left(\frac{s/\alpha-1-\gamma}{\gamma}\right),
\end{aligned}
\]
which is the exponential branch.

The endpoint values agree:
\[
    0=\kappa(\alpha-\alpha),
    \qquad
    \kappa(\alpha(1+\gamma)-\alpha)=\alpha\gamma\kappa,
    \qquad
    \alpha\gamma\kappa
    \exp\!\left(\frac{1+\gamma-1-\gamma}{\gamma}\right)
    =
    \alpha\gamma\kappa.
\]
The formula for $\partial_s\varphi_\gamma^*$ follows by differentiating
the smooth pieces and taking the interval between the one-sided slopes at
$s=\alpha$. At $s=\alpha(1+\gamma)$, the left and right derivatives both
equal $\kappa$, so the subdifferential there is the singleton
$\{\kappa\}$.

It remains to justify the vectorial subdifferential. For $z\neq0$, the
stated formula follows from the chain rule for the convex radial map
$z\mapsto\varphi_\gamma(x,\abs{z})$. At $z=0$, the subgradient condition
is
\[
    \xi\cdot w
    \leq
    \varphi_\gamma(x,\abs{w})
    \qquad\text{for all }w\in\R^N.
\]
If $\abs{\xi}\leq\alpha$, then
$\xi\cdot w\leq\alpha\abs{w}\leq\varphi_\gamma(x,\abs{w})$ for all $w$.
Conversely, if $\abs{\xi}>\alpha$, choosing
$w=\varepsilon\xi/\abs{\xi}$ with $0<\varepsilon\leq\kappa$ gives
\[
    \xi\cdot w
    =
    \varepsilon\abs{\xi}
    >
    \alpha\varepsilon
    =
    \varphi_\gamma(x,\varepsilon),
\]
so the subgradient inequality fails. Hence
$\partial_z\Phi_\gamma(x,0)
=\{\xi\in\R^N:\abs{\xi}\leq\alpha\}$.
\end{proof}

\begin{proposition}[optimality conditions]
    \label{P: optimality logarithmic}
    Let $u_\gamma\in X_\gamma$ be the primal minimizer and let $p_\gamma\in K_\gamma$ be a dual maximizer. Then
    \begin{align*}
        \operatorname{div}p_\gamma
        &=
        u_\gamma-f
        \qquad\qquad\qquad\  \text{in }L^2(\Omega),\\
        p_\gamma(x)
        &\in
        \partial_z\Phi_\gamma
        \left(x,\nabla u_\gamma(x)\right)
        \qquad\text{for a.e. }x\in\Omega,
    \end{align*}
    where $\Phi_\gamma(x,z):=\varphi_\gamma(x,\abs{z})$.
    Equivalently, if $\nabla u_\gamma(x)\neq0$, then
    \begin{align*}
        p_\gamma(x)
        &=
        a(x)
        \frac{\nabla u_\gamma(x)}
        {\abs{\nabla u_\gamma(x)}},\\
        a(x)&\in
        \partial_t\varphi_\gamma
        \left(x,\abs{\nabla u_\gamma(x)}\right),
        \label{D: explicit optimality phi}
        \intertext{pointwise a.e. in $\Omega$, and if $\nabla u_\gamma(x)=0$, then}
        \abs{p_\gamma(x)}&\leq \alpha.
    \end{align*}
\end{proposition}

\begin{proof}
This is \cref{P: optimality conditions Orlicz} applied to $\varphi=\varphi_\gamma$. The radial form follows from the chain rule for the convex radial integrand $\Phi_\gamma(x,z)=\varphi_\gamma(x,\abs{z})$ and \cref{L: logarithmic scalar subdifferentials}.
\end{proof}

\subsection{Dual proximal map}
\label{s:4.3}

For a proper, convex, and lower semicontinuous functional
$H:X\to\R_\infty$ on a Hilbert space $X$, we use the convention
\[
    \operatorname{prox}_{\tau H}(w)
    :=
    \operatorname*{arg\,min}_{v\in X}
    \left\{
        \frac{1}{2\tau}\norm{v-w}_X^2+H(v)
    \right\},
    \qquad \tau>0.
\]
Let $\gamma>0$ and let
$\Phi_\gamma^*(x,p)=\varphi_\gamma^*(x,\abs{p})$ be the pointwise
conjugate of $\Phi_\gamma(x,\cdot)$. We first compute its Euclidean
proximal map.

\begin{lemma}[pointwise dual proximal map]
    \label{L: continuous radial dual prox}
    Let $\sigma>0$. For almost every $x\in\Omega$ and every
    $q\in\R^N$,
    \[
        \operatorname{prox}_{\sigma\Phi_\gamma^*(x,\cdot)}(q)
        =
        \begin{cases}
            0, & q=0,\\[0.5em]
            \displaystyle
            \rho_{\sigma,\gamma}(x,\abs{q})\frac{q}{\abs{q}},
            & q\neq 0,
        \end{cases}
    \]
    where
    \begin{equation}
        \rho_{\sigma,\gamma}(x,r)
        =
        \begin{cases}
            r,
            & 0\leq r\leq\alpha,\\[0.5em]
            \alpha,
            & \alpha<r\leq\alpha+\sigma\kappa(x),\\[0.5em]
            r-\sigma\kappa(x),
            & \alpha+\sigma\kappa(x)<r\leq\alpha(1+\gamma)+\sigma\kappa(x),\\[0.5em]
            \displaystyle
            r-\alpha\gamma\operatorname{W}\!\left(
                \frac{\sigma\kappa(x)}{\alpha\gamma}
                \exp\!\left(\frac{r/\alpha-1-\gamma}{\gamma}\right)
            \right),
            & r>\alpha(1+\gamma)+\sigma\kappa(x).
        \end{cases}
        \label{D: continuous dual radius phi plus}
    \end{equation}
    Here $\operatorname{W}$ is the principal branch of the Lambert
    $W$-function \cite{CorlessEtAl1996}.
\end{lemma}

\begin{proof}
    Fix $x$ and write $\kappa=\kappa(x)$. For $q=0$, the proximal
    minimizer is $0$ since $\Phi_\gamma^*(x,\cdot)$ is nonnegative and
    vanishes at $0$. For $r=\abs{q}>0$, the squared distance to $q$
    on each sphere of radius $\rho$ is minimized at $\rho q/r$.
    Thus the proximal radius minimizes
    \[
        \frac{1}{2\sigma}(\rho-r)^2+\varphi_\gamma^*(x,\rho)
        \qquad\text{over }\rho\geq0.
    \]
    The minimizer is positive because the scalar objective has
    right derivative $-r/\sigma$ at $0$. Its optimality condition is
    \[
        r\in\rho+\sigma\,\partial_s\varphi_\gamma^*(x,\rho).
    \]
    By \eqref{D: explicit dual subgradient phi}, this yields $\rho=r$
    for $0<r\leq\alpha$, $\rho=\alpha$ for
    $\alpha<r\leq\alpha+\sigma\kappa$, and $\rho=r-\sigma\kappa$ for
    $\alpha+\sigma\kappa<r\leq\alpha(1+\gamma)+\sigma\kappa$.

    For $r>\alpha(1+\gamma)+\sigma\kappa$, the radius satisfies
    $\rho>\alpha(1+\gamma)$ and the optimality condition becomes
    \[
        r-\rho=\sigma\kappa
        \exp\!\left(\frac{\rho/\alpha-1-\gamma}{\gamma}\right).
    \]
    Setting $\delta=r-\rho$ gives
    \[
        \frac{\delta}{\alpha\gamma}
        \exp\!\left(\frac{\delta}{\alpha\gamma}\right)
        =
        \frac{\sigma\kappa}{\alpha\gamma}
        \exp\!\left(\frac{r/\alpha-1-\gamma}{\gamma}\right).
    \]
    The right-hand side is positive, so applying the principal
    Lambert $W$-function gives the last branch of
    \eqref{D: continuous dual radius phi plus}.
\end{proof}

On vector fields, regard the conjugate modular as the
extended-real-valued functional
\[
    H_\gamma(p):=\int_\Omega\Phi_\gamma^*(x,p(x))\,dx,
    \qquad p\in L^2(\Omega;\R^N).
\]
It is proper since $H_\gamma(0)=0$, convex by pointwise convexity, and
lower semicontinuous in $L^2$ by Fatou's lemma along a subsequence
converging almost everywhere. For $q\in L^2(\Omega;\R^N)$, let
$P(x)=\operatorname{prox}_{\sigma\Phi_\gamma^*(x,\cdot)}(q(x))$.
The explicit formula gives a measurable field $P$. Comparing the
pointwise proximal objective at $P(x)$ with its value at $0$ yields
\[
    \frac{\abs{P(x)-q(x)}^2}{2\sigma}
    +\Phi_\gamma^*(x,P(x))
    \leq \frac{\abs{q(x)}^2}{2\sigma}
    \qquad\text{for a.e. }x\in\Omega.
\]
Both terms on the left are nonnegative, so $P-q\in L^2(\Omega;\R^N)$
and $H_\gamma(P)<\infty$. In particular, $P\in L^2(\Omega;\R^N)$.
Integrating the pointwise inequality for any competitor
$p\in L^2(\Omega;\R^N)$ shows that
$P=\operatorname{prox}_{\sigma H_\gamma}(q)$.

\begin{figure}[tp]
    \centering
    \includegraphics[width=0.80\textwidth]{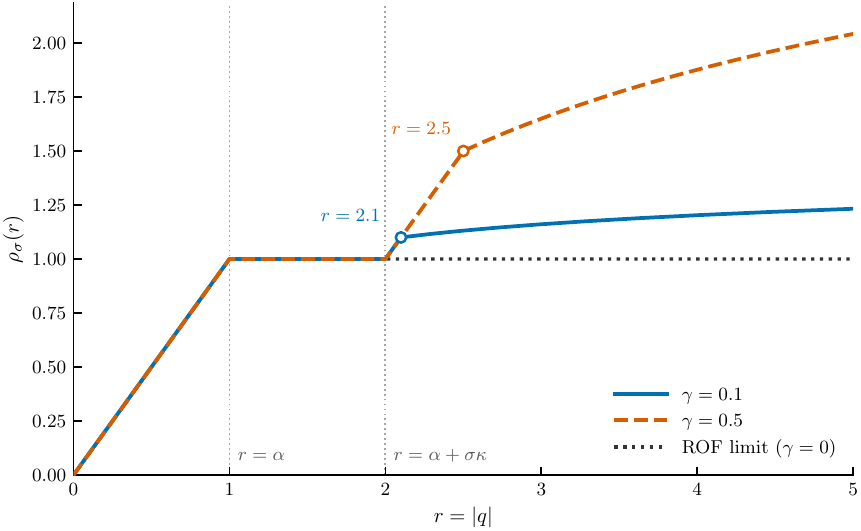}
    \caption{
        Radial profile \(r\mapsto\rho_{\sigma,\gamma}(r)\) of the
        pointwise dual proximal map for
        \(\alpha=\kappa=\sigma=1\) and
        \(\gamma\in\{0.1,0.5\}\).
        The dotted black curve is the limiting ROF projection
        \(\rho_{\sigma,0}(r)=\min\{r,\alpha\}\).
        The vertical dotted lines indicate the common transition points
        \(r=\alpha\) and \(r=\alpha+\sigma\kappa\), while the marked
        points indicate the transition from the linear branch to the
        nonlinear Lambert \(W\) branch.
    }
    \label{fig:dual-prox-gamma}
\end{figure}

As illustrated in \cref{fig:dual-prox-gamma}, the first two branches
agree with the ROF projection onto the ball of radius $\alpha$ for
$r\leq\alpha+\sigma\kappa(x)$. Above this threshold the profile is
linear on an interval of length $\alpha\gamma$, followed by the
Lambert $W$ branch. For fixed $x$ and $r>\alpha+\sigma\kappa(x)$, the
last branch applies for all sufficiently small $\gamma$. Taking the
logarithm of its scalar optimality equation and using
$\rho_{\sigma,\gamma}(x,r)\geq\alpha$ gives
\[
    0\leq \rho_{\sigma,\gamma}(x,r)-\alpha
    \leq
    \alpha\gamma
    \left(1+\log\frac{r-\alpha}{\sigma\kappa(x)}\right).
\]
Together with the first two branches, this proves
$\rho_{\sigma,\gamma}(x,r)\to\min\{r,\alpha\}$ as $\gamma\downarrow0$
for every fixed $r\geq0$ and almost every $x\in\Omega$.

\subsection{\texorpdfstring{$\Gamma$}{Gamma}-convergence as \texorpdfstring{$\gamma\to0$}{gamma to 0}}
\label{s:4.gamma}

We extend $\mathcal{F}_\gamma$ to $L^1(\Omega)$ by setting
\begin{equation}
    \mathcal{F}_\gamma(u):=
    \begin{cases}
        \displaystyle
        \frac12\norm{u-f}_{L^2(\Omega)}^2
        +\int_\Omega\varphi_\gamma(x,\abs{\nabla u})\,dx,
        & u\in X_\gamma,\\[1ex]
        +\infty,
        & u\in L^1(\Omega)\setminus X_\gamma.
    \end{cases}
\end{equation}
The limiting functional is the ROF functional
\begin{equation}
    \mathcal{F}_0(u):=
    \begin{cases}
        \displaystyle
        \frac12\norm{u-f}_{L^2(\Omega)}^2
        +\alpha\abs{Du}(\Omega),
        & u\in BV(\Omega)\cap L^2(\Omega),\\[1ex]
        +\infty,
        & u\in L^1(\Omega)\setminus\left(BV(\Omega)\cap L^2(\Omega)\right).
    \end{cases}
    \label{D: Gamma limit F0}
\end{equation}

\begin{proposition}[compactness]
    \label{P: compactness Gamma convergence}
    Let $\gamma_n\downarrow0$ and let $\{u_n\}_n\subset L^1(\Omega)$ satisfy
    \[
        \sup_n
        \mathcal{F}_{\gamma_n}(u_n)<\infty.
    \]
    Then $\{u_n\}_n$ is bounded in $L^2(\Omega)$ and in $BV(\Omega)$. In particular, there exist a subsequence, not relabelled, and $u\in BV(\Omega)\cap L^2(\Omega)$ such that
    \begin{align*}
        u_n&\to u
        \qquad\text{strongly in }L^1(\Omega),
        \shortintertext{and}
        u_n&\rightharpoonup u
        \qquad\text{weakly in }L^2(\Omega).
    \end{align*}
\end{proposition}

\begin{proof}
The energy bound yields
$\sup_n \norm{u_n-f}_{L^2(\Omega)}<\infty$,
and hence $\{u_n\}_n$ is bounded in $L^2(\Omega)$. Since
$\varphi_{\gamma_n}(x,t)\geq \alpha t$ for a.e. $x\in\Omega$ and
$t\geq0$, we also have
\[
    \alpha\int_\Omega \abs{\nabla u_n}\,dx
    \leq
    \int_\Omega
    \varphi_{\gamma_n}(x,\abs{\nabla u_n})
    \,dx
    \leq C.
\]
Together with
$\norm{u_n}_{L^1(\Omega)}
\leq\abs{\Omega}^{1/2}\norm{u_n}_{L^2(\Omega)}$,
this gives boundedness in $BV(\Omega)$. By compactness of bounded sets in
$BV(\Omega)$ with respect to strong $L^1(\Omega)$-convergence, the compact
embedding $BV(\Omega)\hookrightarrow L^1(\Omega)$ yields, after passing
to a subsequence, strong convergence in $L^1(\Omega)$; see
\cite[Chapter~3]{ambrosio2000functions}. The $L^2$-boundedness yields a
further subsequence converging weakly in $L^2(\Omega)$. The $L^1$-limit
and the weak $L^2$-limit coincide distributionally.
\end{proof}

\begin{theorem}[$\Gamma$-convergence to total variation]
    \label{T: Gamma convergence phi gamma}
    As $\gamma\downarrow0$, the functionals $\mathcal{F}_\gamma:L^1(\Omega)\to[0,\infty]$ $\Gamma$-converge to $\mathcal{F}_0$ with respect to strong convergence in $L^1(\Omega)$. Moreover, the family is equicoercive in $L^1(\Omega)$: if $\gamma_n\downarrow0$ and $\sup_n\mathcal{F}_{\gamma_n}(u_n)<\infty$, then $(u_n)_n$ is precompact in $L^1(\Omega)$, and every $L^1$-cluster point belongs to $BV(\Omega)\cap L^2(\Omega)$.
\end{theorem}

\begin{proof}
We first prove the liminf inequality. Let $\gamma_n\downarrow0$ and suppose that
$u_n\to u$ strongly in $L^1(\Omega)$. If
$\liminf_{n\to\infty}\mathcal{F}_{\gamma_n}(u_n)=+\infty$,
there is nothing to prove. Otherwise, we pass to a subsequence realizing the liminf
and with bounded energy. By \cref{P: compactness Gamma convergence}, the limit
satisfies $u\in BV(\Omega)\cap L^2(\Omega)$, and, after passing to a further
subsequence, $u_n\rightharpoonup u$ weakly in $L^2(\Omega)$. Hence
\[
    \frac12\norm{u-f}_{L^2(\Omega)}^2
    \leq
    \liminf_{n\to\infty}
    \frac12\norm{u_n-f}_{L^2(\Omega)}^2.
\]
Moreover, $\varphi_{\gamma_n}(x,t)\geq\alpha t$, and the lower semicontinuity of
total variation with respect to strong $L^1$-convergence gives
\[
    \alpha\abs{Du}(\Omega)
    \leq
    \liminf_{n\to\infty}
    \alpha\int_\Omega \abs{\nabla u_n}\,dx
    \leq
    \liminf_{n\to\infty}
    \int_\Omega
    \varphi_{\gamma_n}(x,\abs{\nabla u_n})
    \,dx;
\]
see \cite[Chapter~3]{ambrosio2000functions}. Combining the two estimates gives
\[
    \mathcal{F}_0(u)
    \leq
    \liminf_{n\to\infty}
    \mathcal{F}_{\gamma_n}(u_n).
\]

We now prove the limsup inequality. Let $u\in L^1(\Omega)$. If
$\mathcal{F}_0(u)=+\infty$, the constant sequence $u_\gamma=u$
satisfies the limsup inequality. Assume therefore that
$u\in BV(\Omega)\cap L^2(\Omega)$. For $M>0$, let
$u^M=\max\{-M,\min\{u,M\}\}$. Then $u^M\to u$ in $L^2(\Omega)$
as $M\to\infty$, and $\abs{Du^M}(\Omega)\to\abs{Du}(\Omega)$ by
the variation bound for truncations and lower semicontinuity.

For fixed $M$, the approximation on bounded Lipschitz domains in
\cite[Proposition~3.21 and Remark~3.22]{ambrosio2000functions}
gives functions smooth up to the boundary converging strictly in
$BV(\Omega)$ to $u^M$. Clipping these functions to $[-M,M]$ gives
$W^{1,\infty}(\Omega)$ approximants. Clipping preserves convergence
in $L^1$ and can only decrease the total variation, so lower
semicontinuity gives strict convergence of the clipped sequence.
The uniform bound upgrades its $L^1$ convergence to $L^2$ convergence.
A diagonal choice as $M\to\infty$ gives $v_k\in W^{1,\infty}(\Omega)$
such that
\[
    v_k\to u\quad\text{in }L^2(\Omega),
    \qquad
    \int_\Omega\abs{\nabla v_k}\,dx\to\abs{Du}(\Omega).
\]
For fixed $k$, we then have $\nabla v_k\in L^\infty(\Omega;\R^N)$. Since
$\kappa(x)\geq\kappa_->0$,
\[
\begin{aligned}
    0
    &\leq
    \varphi_\gamma(x,\abs{\nabla v_k})
    -
    \alpha\abs{\nabla v_k} =
    \alpha\gamma\abs{\nabla v_k}
    \log^+\left(
    \frac{\abs{\nabla v_k}}{\kappa(x)}
    \right) \\
    &\leq
    \alpha\gamma
    \norm{\nabla v_k}_{L^\infty(\Omega)}
    \log^+\left(
    \frac{\norm{\nabla v_k}_{L^\infty(\Omega)}}{\kappa_-}
    \right).
\end{aligned}
\]
The right-hand side is a constant depending on $k$, multiplied by $\gamma$. Therefore
\[
    \lim_{\gamma\downarrow0}
    \int_\Omega
    \varphi_\gamma(x,\abs{\nabla v_k})
    \,dx
    =
    \alpha\int_\Omega \abs{\nabla v_k}\,dx.
\]
Consequently,
\[
    \lim_{k\to\infty}
    \lim_{\gamma\downarrow0}
    \mathcal{F}_\gamma(v_k)
    =
    \mathcal{F}_0(u).
\]
A standard diagonal argument for recovery sequences, see
\cite[Chapter~1]{dal2012introduction}, gives $u_\gamma=v_{k(\gamma)}$ such that
$u_\gamma\to u$ strongly in $L^1(\Omega)$ and
\[
    \limsup_{\gamma\downarrow0}
    \mathcal{F}_\gamma(u_\gamma)
    \leq
    \mathcal{F}_0(u).
\]
This proves the $\Gamma$-convergence.
\end{proof}

\begin{remark}[convergence of minimizers]
The equicoercivity statement in \cref{T: Gamma convergence phi gamma} is the compactness property needed to pass from $\Gamma$-convergence to convergence of minimizers. Let $u_\gamma$ be the unique minimizer of $\mathcal{F}_\gamma$. Since
\[
    \mathcal{F}_\gamma(u_\gamma)
    \leq
    \mathcal{F}_\gamma(0)
    =
    \frac12\norm{f}_{L^2(\Omega)}^2,
\]
the family $(u_\gamma)_\gamma$ has $L^1$-convergent subsequences as $\gamma\downarrow0$. By the fundamental theorem of $\Gamma$-convergence, every cluster point is a minimizer of $\mathcal{F}_0$. Since $\mathcal{F}_0$ is strictly convex because of the quadratic fidelity term, its minimizer is unique. Therefore the whole family $u_\gamma$ converges in $L^1(\Omega)$ to the unique minimizer of $\mathcal{F}_0$. The same argument applies to asymptotic minimizers satisfying
\[
    \mathcal{F}_\gamma(u_\gamma)
    \leq
    \inf \mathcal{F}_\gamma+o(1).
\]
\end{remark}

The scale function $\kappa$ sets the gradient threshold for the
logarithmic correction at positive $\gamma$. The limit as
$\gamma\downarrow0$ is the unweighted ROF functional for every
admissible $\kappa$.

\section{Numerical illustrations}
\label{s:5}

We conclude with two numerical illustrations of the finite-$\gamma$
behavior of the scaled logarithmic regularizer. The purpose is to
demonstrate the effect of the logarithmic correction in a controlled
example and to compare it with standard denoising models.

\subsection{Discretization and experimental setup}
\label{s:5.1}

All images consist of $n\times n$ pixels with $n=256$. We denote by
$D$ the unscaled forward-difference gradient, with outward differences
set to zero at the image boundary, and use the Euclidean norm on the
two gradient components. The discrete Orlicz problem with a positive scale $\kappa_{ij}$ at
pixel $(i,j)$ is
\begin{equation}
    \min_{u\in\mathbb R^{n\times n}}
    \frac12\sum_{i,j}(u_{ij}-f_{ij})^2
    +
    \alpha\sum_{i,j}\abs{(Du)_{ij}}
    \left(
        1+\gamma\log^+
        \frac{\abs{(Du)_{ij}}}{\kappa_{ij}}
    \right).
    \label{E: discrete Orlicz model}
\end{equation}
We compare \eqref{E: discrete Orlicz model} with the discrete ROF model,
obtained by setting $\gamma=0$, and with quadratic $H^1$ regularization,
whose regularization term is
\[
    \frac{\alpha}{2}\sum_{i,j}\abs{(Du)_{ij}}^2.
\]

Let $D^*$ denote the Euclidean adjoint of $D$ and set
$\operatorname{div}_D:=-D^*$. The discrete dual regularization term is
the sum of the pointwise conjugates over the pixels. Its proximal map
$P_\sigma$ therefore acts independently at each pixel. For the Orlicz
model, we evaluate the formula in \cref{L: continuous radial dual prox}
with $N=2$ and scale $\kappa_{ij}$ at pixel $(i,j)$. For ROF, it is
the pixelwise projection onto the Euclidean ball of radius $\alpha$.

For both models we use the primal--dual iteration of
\cite{ChambollePock2011} (see also \cite[Chapter~7]{NAO:2020}),
\begin{equation}
    \label{E: discrete pdhg}
    \begin{aligned}
        p^{k+1}
        &=
        P_\sigma(p^k+\sigma D\bar u^k),\\
        u^{k+1}
        &=
        \frac{u^k+\tau\operatorname{div}_D p^{k+1}+\tau f}{1+\tau},\\
        \bar u^{k+1}
        &=
        u^{k+1}+\theta(u^{k+1}-u^k),
    \end{aligned}
\end{equation}
starting from $u^0=\bar u^0=f$ and $p^0=0$. We take $\theta=1$ and
$\tau=\sigma=0.99/\sqrt{8}$, so $\tau\sigma\|D\|^2<1$ by the
bound $\|D\|^2\leq8$. We run $5000$ iterations for both models.
The quadratic $H^1$ problem is solved by conjugate gradients.

The clean images have intensity scale $[0,1]$, and the image
figures use the same display range. For discrete images
$u,v\in\R^{n\times n}$, we report the mean squared error
$\operatorname{MSE}(u,v):=n^{-2}\sum_{i,j=1}^n
\abs{u_{ij}-v_{ij}}^2$
and the peak signal-to-noise ratio
$\operatorname{PSNR}(u,v):=-10\log_{10}(\operatorname{MSE}(u,v))$
in decibels~(dB). Both metrics are evaluated on the full,
unclipped arrays.

For the constant-scale experiments, we set $\kappa_{ij}=\kappa$ and
estimate $\kappa$ from the noisy datum by
\begin{equation}
    \kappa
    =
    \frac12 Q_{0.95}\!\left(
        \left(
            \left\lvert \bigl(D(G_1*f)\bigr)_{ij}\right\rvert_2
        \right)_{(i,j):\,\lvert(D(G_1*f))_{ij}\rvert_2>0}
    \right),
    \label{E: numerical kappa}
\end{equation}
where $G_1$ denotes Gaussian smoothing with standard deviation one
pixel and $Q_{0.95}$ denotes the empirical $95$th percentile of the
positive pixelwise gradient magnitudes. For the grid-based parameter
selection, we minimize the mean squared error against the clean
reference independently for each method over its prescribed grid.

The source code, notebooks, and numerical outputs are available
in the archived software release \cite{Unterberger2026CUOrlicz}.

\subsection{Blurred-disk experiment}
\label{s:5.2}

On the square $[-1,1]^2$, we sample the radially symmetric phantom
\[
    u_\eta(x)
    =
    \frac{1}{
        1+\exp\!\left(
            \dfrac{\abs{x}-0.35}{\eta}
        \right)
    }
\]
on an equally spaced $n\times n$ grid. At each grid point $x_{ij}$,
the noisy datum is
\[
    (f_\eta)_{ij}=u_\eta(x_{ij})+0.2\,\xi_{ij},
\]
where the $\xi_{ij}$ are independent standard Gaussian random variables.
We use the same noise realization for every transition width $\eta$
and pass the unclipped arrays to the solvers. The scale estimates
from \eqref{E: numerical kappa} range from $0.04695$ to $0.05037$.

\Cref{tab:blurred-disk-grids} lists the final grids used after an
exploratory sweep. For each $\eta$, we select $\alpha$ by minimum
MSE for ROF and quadratic $H^1$. For Orlicz, we minimize MSE over
the Cartesian product of the shared $\alpha$ grid and the positive
$\gamma$ grid. \Cref{tab:blurred-disk} reports the selected parameters
and their MSEs, and \cref{fig:blurred-disk-comparison} shows the
corresponding reconstructions for three transition widths.

\begin{table}[b]
    \caption{Final parameter grids for the blurred-disk experiment.}
    \label{tab:blurred-disk-grids}
    \centering
    \small
    \setlength{\tabcolsep}{5pt}
    \begin{tabular}{cccc}
        \toprule
        $\eta$ & $\alpha$ for ROF and Orlicz
        & $\gamma$ for Orlicz & $\alpha$ for $H^1$ \\
        \midrule
        $0.001$ & $\{0.2,0.3,0.4\}$ & $\{0.005,0.01,0.05,0.1\}$ & $\{2,5,10\}$ \\
        $0.010$ & $\{0.2,0.3,0.4\}$ & $\{0.05,0.1,0.15,0.2\}$   & $\{2,5,10\}$ \\
        $0.035$ & $\{0.2,0.3,0.5,0.7\}$ & $\{0.3,0.6,0.8\}$     & $\{5,10,20\}$ \\
        $0.050$ & $\{0.2,0.3,0.5\}$ & $\{0.6,0.8,1.0\}$         & $\{10,20,40\}$ \\
        $0.075$ & $\{0.3,0.5,0.7\}$ & $\{0.6,0.8,1.0,1.25\}$    & $\{10,20,40\}$ \\
        \bottomrule
    \end{tabular}
\end{table}

\pgfplotstableread[
    col sep=comma
]{tables/blurred_disk_metrics.csv}\blurreddiskmetrics

\newcommand{\bdnumber}[3][]{%
    \pgfplotstablegetelem{#2}{#3}\of\blurreddiskmetrics%
    \pgfmathprintnumber[#1]{\pgfplotsretval}%
}

\newcommand{\bdeta}[1]{%
    \bdnumber[fixed,precision=3,zerofill]{#1}{eta}%
}

\newcommand{\bdalpha}[1]{%
    \bdnumber[fixed,precision=2]{#1}{alpha}%
}

\newcommand{\bdgamma}[1]{%
    \pgfplotstablegetelem{#1}{gamma}\of\blurreddiskmetrics%
    \edef\bdgammavalue{\pgfplotsretval}%
    \pgfmathtruncatemacro{\bdgammaprecision}{%
        \bdgammavalue < 0.01 ? 3 : 2%
    }%
    \pgfmathprintnumber[
        fixed,
        precision=\bdgammaprecision,
        zerofill
    ]{\bdgammavalue}%
}

\newcommand{\bdmse}[1]{%
    \pgfplotstablegetelem{#1}{MSE}\of\blurreddiskmetrics%
    \pgfmathprintnumber[
        fixed,
        precision=4,
        zerofill
    ]{%
        \fpeval{10000*(\pgfplotsretval)}%
    }%
}

\newcommand{\bdrow}[3]{%
    $\bdeta{#1}$
    & $\bdalpha{#1}$ & $\bdmse{#1}$
    & $(\bdalpha{#2},\bdgamma{#2})$ & $\bdmse{#2}$
    & $\bdalpha{#3}$ & $\bdmse{#3}$%
}

\begin{table}[bp]
    \caption{MSE-selected parameters and corresponding mean squared errors for
    the blurred-disk experiment.}
    \label{tab:blurred-disk}
    \centering
    \small
    \setlength{\tabcolsep}{4pt}
    \begin{tabular}{c cc cc cc}
        \toprule
        & \multicolumn{2}{c}{ROF}
        & \multicolumn{2}{c}{Orlicz}
        & \multicolumn{2}{c}{$H^1$} \\
        \cmidrule(lr){2-3}
        \cmidrule(lr){4-5}
        \cmidrule(lr){6-7}
        $\eta$
        & $\alpha$ & $10^4\mathrm{MSE}$
        & $(\alpha,\gamma)$ & $10^4\mathrm{MSE}$
        & $\alpha$ & $10^4\mathrm{MSE}$ \\
        \midrule
        \bdrow{0}{1}{2}   \\
        \bdrow{3}{4}{5}   \\
        \bdrow{6}{7}{8}   \\
        \bdrow{9}{10}{11} \\
        \bdrow{12}{13}{14}\\
        \bottomrule
    \end{tabular}
\end{table}


\begin{figure}[tp]
    \centering
    \begin{subfigure}[t]{\textwidth}
        \centering
        \includegraphics[
            width=\linewidth
        ]{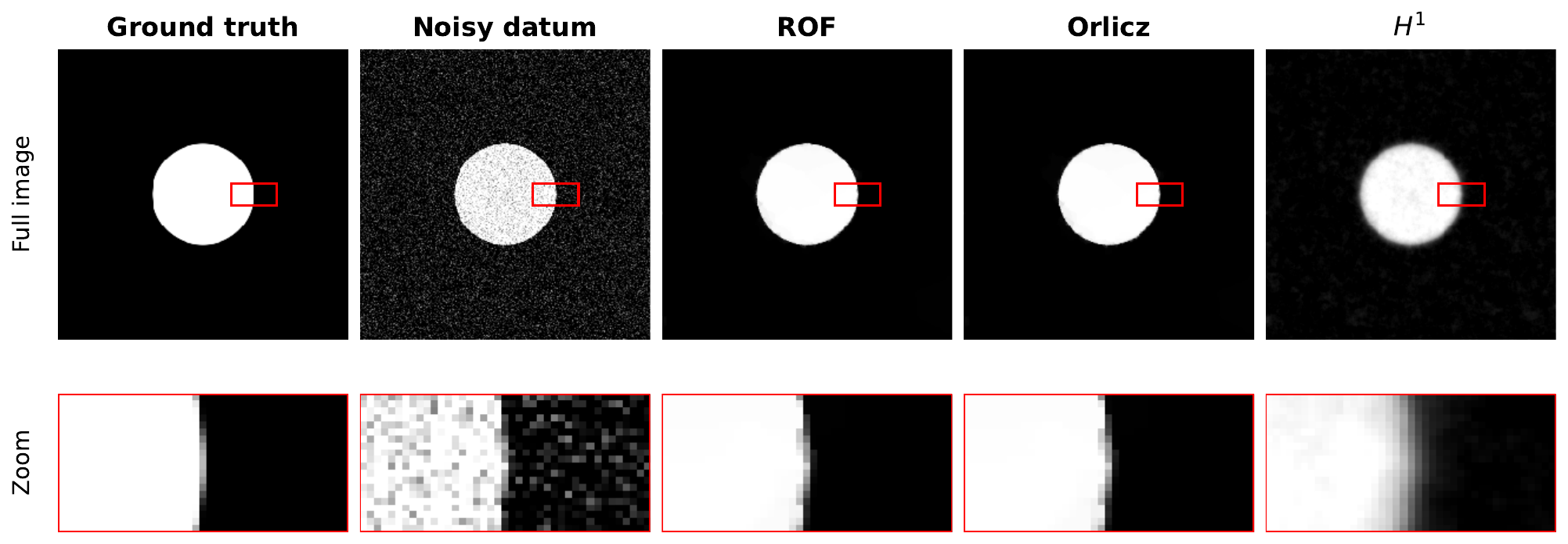}
        \caption{$\eta=0.001$.}
        \label{fig:blurred-disk-sharp}
    \end{subfigure}

    \vspace{0.6em}

    \begin{subfigure}[t]{\textwidth}
        \centering
        \includegraphics[
            width=\linewidth
        ]{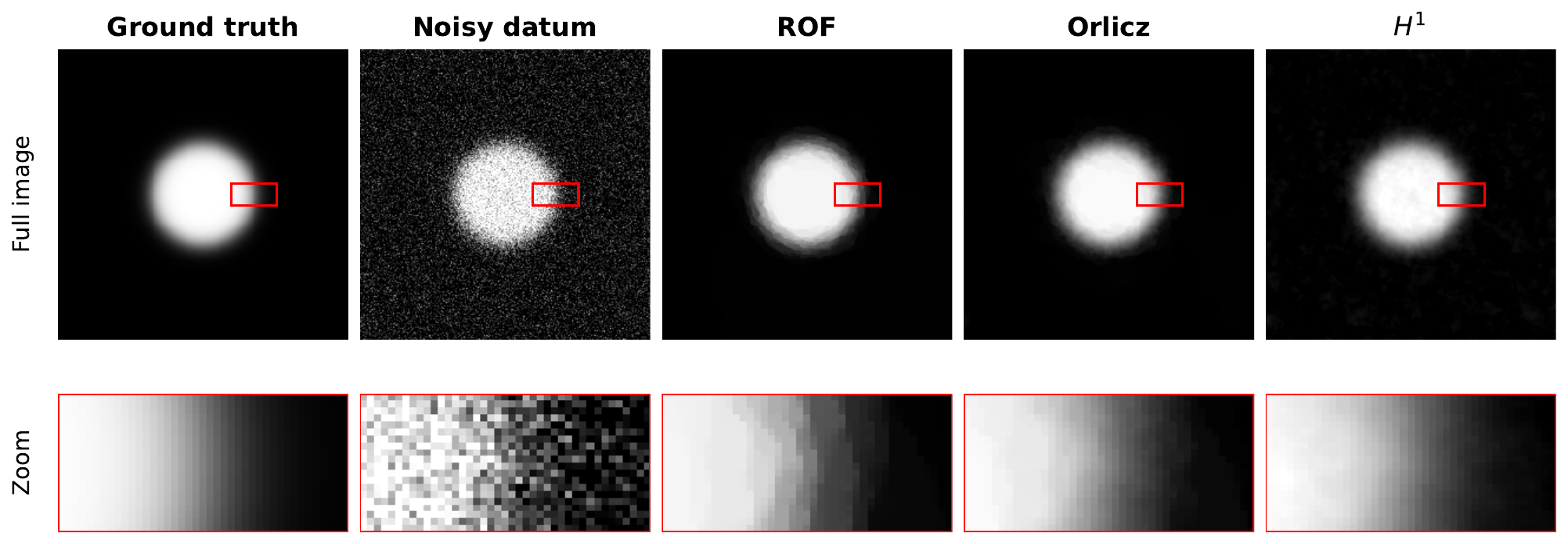}
        \caption{$\eta=0.035$.}
        \label{fig:blurred-disk-intermediate}
    \end{subfigure}

    \vspace{0.6em}

    \begin{subfigure}[t]{\textwidth}
        \centering
        \includegraphics[
            width=\linewidth
        ]{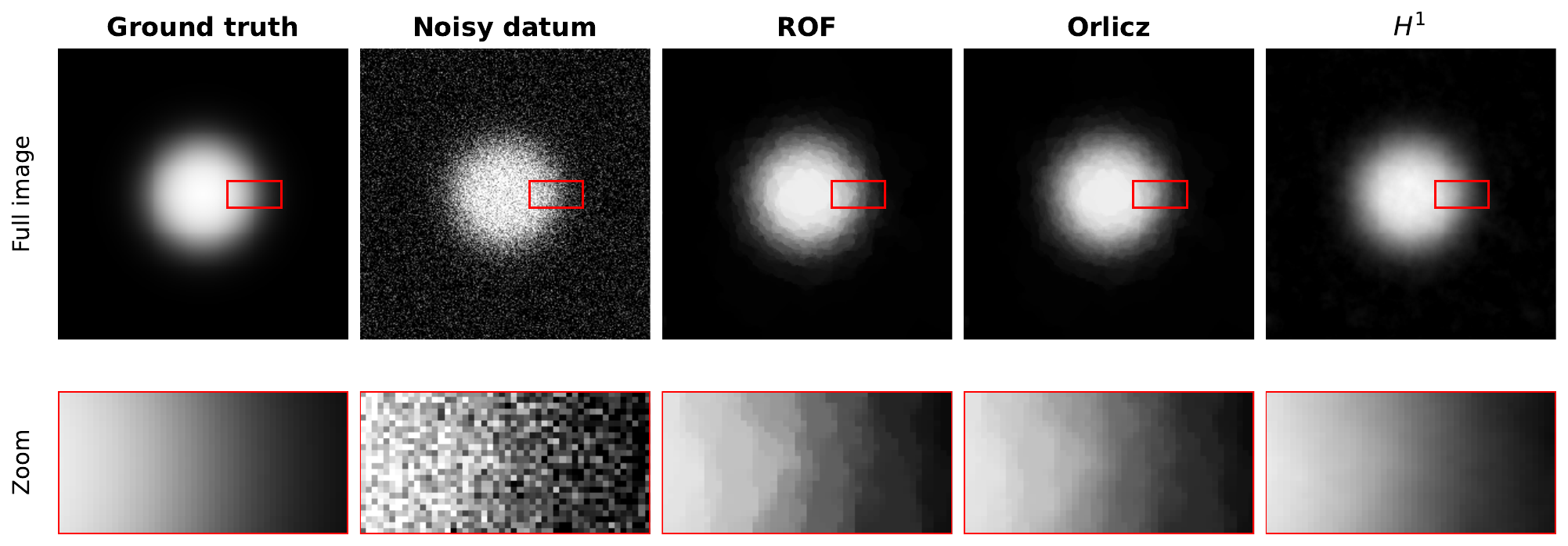}
        \caption{$\eta=0.075$.}
        \label{fig:blurred-disk-diffuse}
    \end{subfigure}

    \caption{
        MSE-selected reconstructions of blurred disks with increasing
        transition width. In each panel, the columns show the ground
        truth, noisy datum, ROF, Orlicz, and quadratic $H^1$
        reconstructions. The upper row contains the full images and the
        lower row enlarges the red transition region.
    }
    \label{fig:blurred-disk-comparison}
\end{figure}

At $\eta=0.001$, the phantom's transition from intensity $0.9$ to
$0.1$ has radial width $2\eta\log 9$, below the grid spacing
$2/(n-1)$. The selected Orlicz reconstruction with $\gamma=0.005$
looks very similar to ROF in \cref{fig:blurred-disk-sharp}, with an
MSE about $1.9\%$ higher. For the four wider disks, the selected Orlicz
reconstructions have full-image MSEs about $13\%$ to $43\%$ lower
than ROF. At $\eta=0.035$ and $\eta=0.075$, the zooms show less
pronounced concentric staircasing in the Orlicz reconstructions.
The $\gamma$ grids vary with $\eta$, so each selected value describes
the best tested pair for that particular phantom.

\subsection{Natural-image comparison}
\label{s:5.3}

We add Gaussian noise with standard deviation $0.12$ to the standard
Cameraman image on the intensity scale $[0,1]$. For oracle tuning,
ROF and Orlicz use the common grid
\[
    \alpha\in\{0.03,0.05,0.075,0.1,0.15,0.2,0.3,0.45\}.
\]
The Orlicz $\gamma$ grid is $\{0.01,0.05,0.1,0.2,0.3\}$, and the
$H^1$ $\alpha$ grid is $\{0.5,1,2,5,10,20\}$.

\Cref{fig:cameraman-comparison} shows the MSE-selected reconstructions.
ROF selects $\alpha=0.1$ and attains MSE $1.491\cdot10^{-3}$
(PSNR $28.27$ dB). The Orlicz model selects
$(\alpha,\gamma)=(0.1,0.01)$ with $\kappa\approx0.05425$ and attains
MSE $1.499\cdot10^{-3}$ (PSNR $28.24$ dB). The quadratic $H^1$ model
selects $\alpha=1$ and attains MSE $2.586\cdot10^{-3}$ (PSNR $25.87$ dB).
The ROF and Orlicz errors are very close in this example, and
$\gamma=0.01$ is the smallest positive value tested in the oracle grid.

\begin{figure}[tp]
    \centering
    \includegraphics[width=\textwidth]{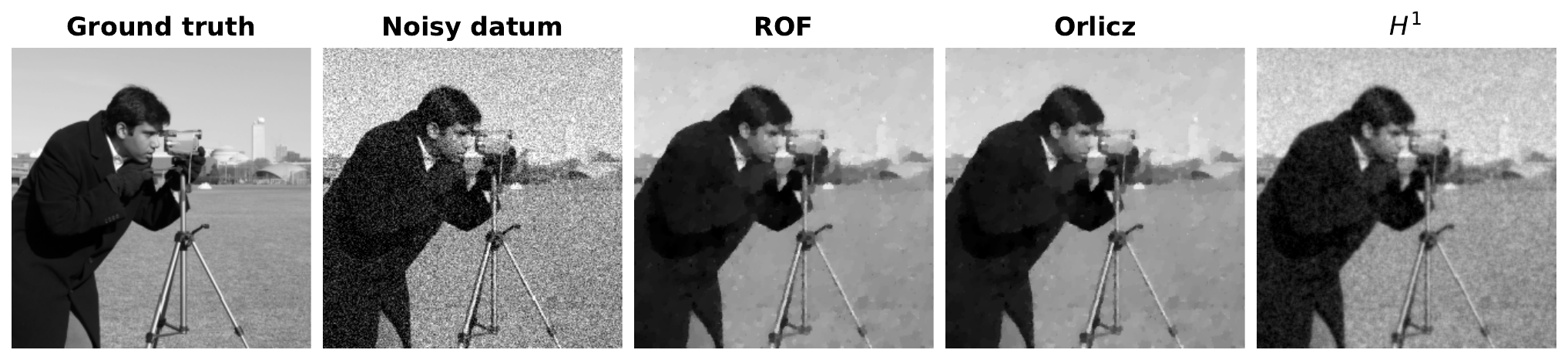}
    \caption{
        Cameraman denoising with additive Gaussian noise of standard
        deviation $0.12$. The panels show the ground truth, noisy datum,
        and MSE-selected ROF, Orlicz, and quadratic $H^1$
        reconstructions from left to right.
    }
    \label{fig:cameraman-comparison}
\end{figure}

To isolate the effect of $\gamma$, we fix $\alpha=0.1$ and use the
same noisy datum and estimated scale $\kappa\approx0.05425$.
\Cref{fig:cameraman-gamma-comparison} shows the ROF reconstruction and
the Orlicz reconstructions at $\gamma\in\{0.01,0.1,0.5,1\}$.
The reconstruction at $\gamma=0.01$ is visually close to ROF, while
larger values progressively soften edges and fine detail in this
example. The full-image root mean squared difference between the
Orlicz and ROF reconstructions increases from $6.06\cdot10^{-4}$ at
$\gamma=0.01$ to $2.69\cdot10^{-2}$ at $\gamma=1$.

\begin{figure}[tp]
    \centering
    \includegraphics[width=\linewidth]
        {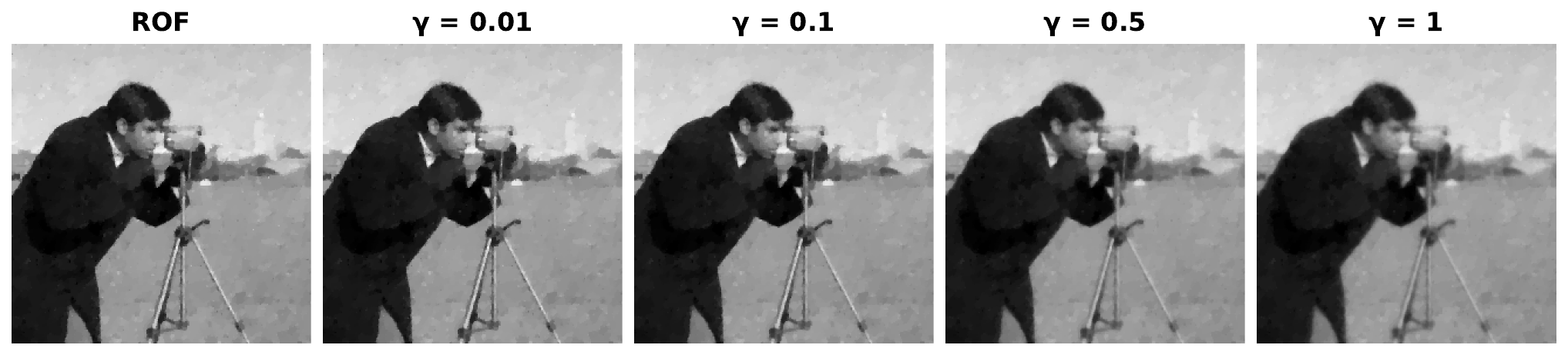}
    \caption{%
        Influence of $\gamma$ on Cameraman denoising at fixed
        $\alpha=0.1$ and $\kappa\approx0.05425$. The panels show ROF
        and the Orlicz reconstructions at $\gamma=0.01$, $0.1$, $0.5$,
        and $1$ from left to right.
    }
    \label{fig:cameraman-gamma-comparison}
\end{figure}

We also compare the fixed Orlicz reconstruction at $\gamma=1$ with
a quadratic $H^1$ reconstruction. We choose the $H^1$ weight by
approximately minimizing the full-image MSE to the Orlicz
reconstruction, using a logarithmic parameter search followed by
local refinement. This yields $\alpha\approx3.15$ and a root mean
squared difference of $1.295\cdot10^{-2}$ between the reconstructions.
In \cref{fig:cameraman-h1-comparison}, the Orlicz reconstruction shows
fewer small-scale intensity fluctuations in the sky, and the two
reconstructions also differ along the outline of the cameraman.

\begin{figure}[tp]
    \centering
    \includegraphics[width=\textwidth]{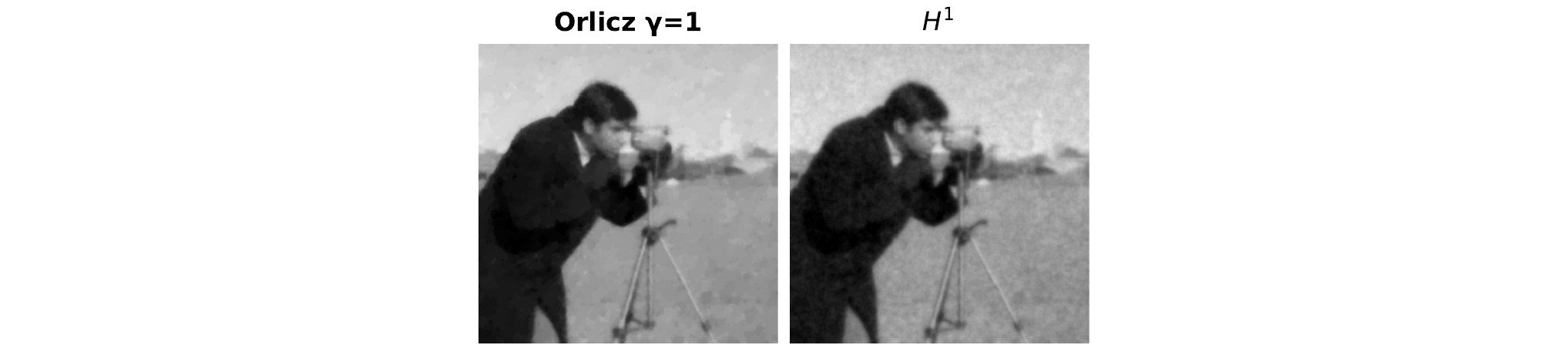}
    \caption{
        Cameraman denoising with Orlicz regularization at
        $\alpha=0.1$, $\gamma=1$, and $\kappa\approx0.05425$
        (left), and matched quadratic $H^1$ regularization
        at $\alpha\approx3.15$ (right).
    }
    \label{fig:cameraman-h1-comparison}
\end{figure}

Finally, we illustrate a spatially varying scale on the same noisy
Cameraman datum. We estimate the local gradient magnitude after
Gaussian smoothing to reduce the influence of noise. Taking this
magnitude as the threshold gives a larger range of gradients over
which the density coincides with TV near strong estimated edges.
To keep the threshold from becoming too small in nearly flat
regions, we use one tenth of the scalar estimate $\kappa$ from
\eqref{E: numerical kappa} as a fixed relative floor, with $10^{-8}$
as an absolute numerical safeguard. We thus set
\[
    \kappa_{ij}
    =
    \max\left\{
        10^{-8},\,
        0.1\kappa,\,
        \left\lvert \bigl(D(G_1*f)\bigr)_{ij}\right\rvert_2
    \right\}.
\]
For this datum, $0.1\kappa$ sets the floor, and the resulting values
of $\kappa_{ij}$ range from approximately $0.00542$ to $0.33332$.
\Cref{fig:cameraman-spatial-kappa} compares the spatial-$\kappa$
Orlicz reconstruction at $\alpha=\gamma=0.1$ with TV and quadratic
$H^1$ regularization. The full-image RMSEs against the clean reference
are $0.03861$ for TV, $0.03911$ for spatial-$\kappa$ Orlicz, and
$0.05085$ for quadratic $H^1$. The spatial Orlicz reconstruction is
visually close to TV, and both show less background grain than the
quadratic $H^1$ reconstruction.
At the same $\alpha=\gamma=0.1$, the scalar scale yields an RMSE of
$0.03988$ (see \cref{fig:cameraman-gamma-comparison}).

\begin{figure}[tp]
    \centering
    \includegraphics[width=\linewidth]{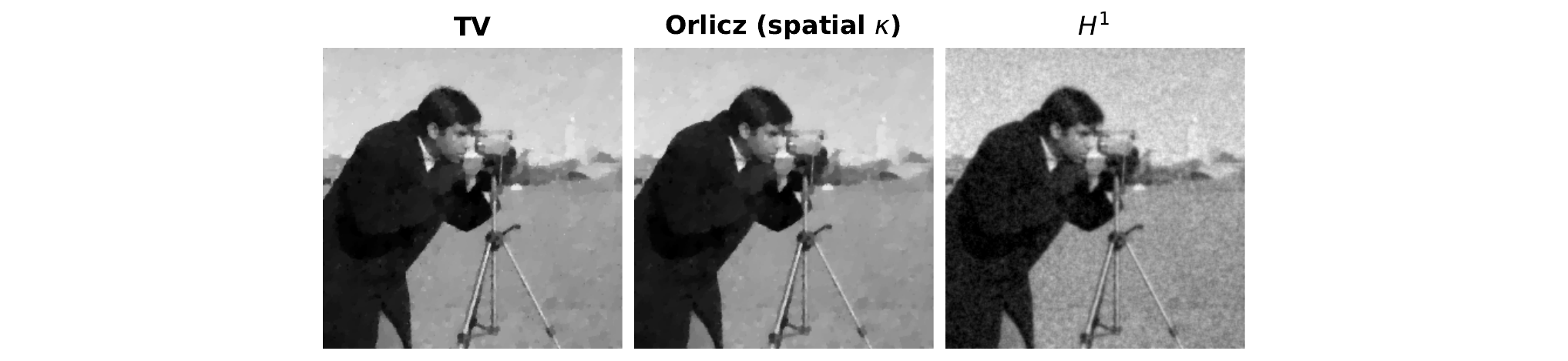}
    \caption{%
        Cameraman reconstructions on the same noisy datum. The panels
        show TV ($\alpha=0.1$), Orlicz with spatially varying $\kappa$
        ($\alpha=\gamma=0.1$), and quadratic $H^1$ ($\alpha=1$) regularization, 
        from left to right.
    }
    \label{fig:cameraman-spatial-kappa}
\end{figure}


\section{Conclusion}

The scaled $L\log L$ model provides a variationally consistent
approximation of total variation denoising with additional control over
the penalization of large gradients. The Orlicz framework establishes
well-posedness and Fenchel duality despite the possible lack of
reflexivity, while the explicit dual proximal map makes the model
accessible to primal--dual splitting methods. The $\Gamma$-convergence
result and convergence of minimizers give a rigorous connection to ROF
as the logarithmic correction vanishes.

Within the tested parameter grids, the Orlicz reconstructions attain
lower full-image MSE than ROF on the four wider blurred disks, and the
displayed diffuse transitions show less pronounced staircasing. In
the Cameraman sweep, increasing $\gamma$ at fixed $\alpha$ and
$\kappa$ softens edges and fine detail. 
On Cameraman, the spatially varying scale delays the onset of the
logarithmic correction near strong estimated edges and yields a
reconstruction with an RMSE close to that of ROF.

The theoretical framework permits spatially dependent integrands beyond
the particular logarithmic density studied here.
Further directions include extending the analysis and numerical
realization to inverse problems such as deblurring, and developing
parameter-choice rules that do not require access to a clean reference
image.


\section*{Statements and Declarations}

\subsection*{Author Contributions}

Both authors contributed to the conception and mathematical development of
the study. Tobias Unterberger implemented and carried out the numerical
experiments. Both authors contributed to writing and revising the manuscript,
and both authors read and approved the final version.

\subsection*{Data and Code Availability}

The source code, notebooks, and numerical outputs supporting the numerical
experiments are openly available in the archived Zenodo release
\emph{CU Orlicz numerics}, Version 0.1.1,
\url{https://doi.org/10.5281/zenodo.22799164}.

\subsection*{LLM use}

ChatGPT and Claude Opus 5.5 were used during the preparation of this manuscript
as an interactive assistant for drafting and revising text, improving the
structure and exposition of mathematical arguments, suggesting alternative
formulations and intermediate derivations, and assisting with the numerical
presentation. Codex was used for developing and debugging parts of the
numerical code. The mathematical ideas, formulation of the problem, main
results, and underlying proof strategies were developed by the authors. All
AI-assisted text, mathematical derivations, code, and numerical results were
reviewed, verified, and revised by the authors, who take full responsibility
for the content of the manuscript and the accompanying code.

\bibliographystyle{jnsao}
\bibliography{references}

\end{document}